\documentclass[12pt]{amsart}
\usepackage{mathtools,amsthm,mathrsfs,txfonts,yfonts,titletoc}
\usepackage[all]{xy}
\usepackage[
	colorlinks,anchorcolor=blue,citecolor=blue,linkcolor=blue,urlcolor =blue,
	bookmarksdepth=2
]{hyperref}
\newtheorem*{thm*}{Theorem}
\newtheorem*{prop*}{Proposition}
\newtheorem*{cor*}{Corollary}
\newtheorem{thm}{Theorem}[section]
\newtheorem{defn}[thm]{Definition}
\newtheorem{prop}[thm]{Proposition}
\newtheorem{cor}[thm]{Corollary}
\newtheorem{lemma}[thm]{Lemma}

\numberwithin{equation}{section}

\newcommand{\Z}{\varmathbb{Z}}
\newcommand{\C}{\varmathbb{C}}
\newcommand{\I}{\mathcal{I}}
\newcommand{\fO}{\mathcal{O}}
\newcommand{\lP}{\varmathbb{P}}
\newcommand{\F}{\varmathbb{F}}
\newcommand{\lG}{\varmathbb{G}}
\newcommand{\lV}{\varmathbb{V}}

\begin{document}
\title{\rm New cubic fourfolds with odd-degree \\ unirational parametrizations}
\author{Kuan-Wen Lai}
\address{Department of Mathematics, Brown University, Providence, RI}
\email{kwlai@math.brown.edu}

\begin{abstract}
We prove that the moduli space of cubic fourfolds $\mathcal{C}$ contains a divisor $\mathcal{C}_{42}$ whose general member has a unirational parametrization of degree 13. This result follows from a thorough study of the Hilbert scheme of rational scrolls and an explicit construction of examples. We also show that $\mathcal{C}_{42}$ is uniruled.
\end{abstract}
\maketitle

\DeclareRobustCommand{\gobblefour}[4]{}
\newcommand*{\SkipTocEntry}{\addtocontents{toc}{\gobblefour}}
\setcounter{tocdepth}{1}
\tableofcontents
\parskip = 5pt

\section*{Introduction}
Let $X$ be a smooth projective variety of dimension $n$ over $\C$. We say that $X$ has a \emph{degree $\varrho$ unirational parametrization} if there is a dominant rational map $\rho:\lP^n\dashrightarrow X$ with $\deg\rho=\varrho$. Such a parametrization implies that the smallest positive integer $N$ which allows the rational equivalence
\begin{equation}\label{decDiag}
	N\Delta_X\equiv N\{x\times X\}+Z\quad\mbox{in}\quad {\rm CH}^n(X\times X)
\end{equation}
would divide $\varrho$, where $x\in X$ and $Z$ is a cycle supported on $X\times Y$ for some divisor $Y\subset X$. The relation (\ref{decDiag}) for arbitrary integer $N$ is called a \emph{decomposition of the diagonal} of $X$, and it is called an \emph{integral} decomposition of the diagonal if $N=1$. (\cite{BS83}. See also \cite[Chap. 3]{Voi14}.)

This paper studies the unirationality of \emph{cubic fourfolds}, i.e. smooth cubic hypersurfaces in $\lP^5$ over $\C$. Let ${\rm Hdg}^4(X,\Z):=H^4(X,\Z)\cap H^2(\Omega_X^2)$ be the group of integral Hodge classes of degree 4 for a cubic fourfold $X$. In the coarse moduli space of cubic fourfolds $\mathcal{C}$, the Noether-Lefschetz locus
\[
	\left\{X\in\mathcal{C}\,:\,{\rm rk}\left({\rm Hdg}^4(X,\Z)\right)\geq 2\right\}
\]
is a countably infinite union of irreducible divisors $\mathcal{C}_d$ indexed by $d\geq8$ and $d\equiv 0,2$ (mod 6). Here $\mathcal{C}_d$ consists of the \emph{special cubic fourfolds} which admit a rank-2 saturated sublattice of discriminant $d$ in ${\rm Hdg}^4(X,\Z)$ \cite{Has00}. Because the integral Hodge conjecture is valid for cubic fourfolds \cite[Th. 1.4]{Voi13}, $X\in\mathcal{C}$ is special if and only if there is an algebraic surface $S\subset X$ not homologous to a complete intersection.

Voisin \cite[Th. 5.6]{Voi15} proves that a special cubic fourfold of discriminant $d\equiv 2$ (mod 4) admits an integral decomposition of the diagonal. Because every cubic fourfold has a unirational parametrization of degree 2 \cite[Example 18.19]{Har95}, it is natural to ask whether they have odd degree unirational parametrizations.

For a general $X\in\mathcal{C}_d$ with $d=14$, 18, 26, 30, and 38, the examples constructed by Nuer \cite{Nue15} combined with an algorithm by Hassett \cite[Prop. 38]{Has16} support the expectation. In this paper, we improve the list by solving the case $d=42$.
\begin{thm}\label{mainThm}
	A generic $X\in\mathcal{C}_{42}$ has a degree 13 unirational parametrization.
\end{thm}

Recall that a variety $Y$ is \emph{uniruled} if there is a variety $Z$ and a dominant rational map $Z\times\lP^1\dashrightarrow Y$ which doesn't factor through the projection to $Z$. As a byproduct of the proof of Theorem \ref{mainThm}, we also prove that
\begin{thm}
	$\mathcal{C}_{42}$ is uniruled.
\end{thm}

\subsection*{Strategy of Proof}
When $d=2(n^2+n+1)$ with $n\geq2$ and $X\in\mathcal{C}_d$ is general, the Fano variety of lines $F_1(X)$ is isomorphic to the Hilbert scheme of two points $\Sigma^{[2]}$, where $\Sigma$ is a K3 surface polarized by a primitive ample line bundle of degree $d$. \cite[Th. 6.1.4]{Has00}

The isomorphism $F_1(X)\cong\Sigma^{[2]}$ implies that $X$ contains a family of two-dimensional rational scrolls parametrized by $\Sigma$. Indeed, the divisor $\Delta\subset\Sigma^{[2]}$ parametrizing the non-reduced subschemes can be naturally identified as the projectivization of the tangent bundle of $\Sigma$. Each fiber of this $\lP^1$-bundle induces a smooth rational curve in $F_1(X)$ through the isomorphism, hence corresponds to a rational scroll in $X$.

Let $S\subset X$ be one of such scrolls. Since $S$ is rational, its symmetric square $W={\rm Sym}^2S$ is also rational. A generic element $s_1+s_2\in W$ spans a line $l(s_1,s_2)$ not contained in $X$, so there is a rational map
\[\begin{array}{cccc}
	\rho:&W&\dashrightarrow&X\\
		&s_1+s_2&\mapsto&x
\end{array}\]
where $l(s_1,s_2)\cap X=\{s_1,s_2,x\}$. By \cite[Prop. 38]{Has16}, this map becomes a unirational parametrization if $S$ has isolated singularities. Moreover, its degree is odd as long as $4\nmid d$.

Discriminant $d=42$ corresponds to the case $n=4$ above. Note that $4\nmid d=42$. Thus a generic $X\in\mathcal{C}_{42}$ admits an odd degree unirational parametrization once we prove that
\begin{thm}\label{mainStep}
	A generic $X\in\mathcal{C}_{42}$ contains a degree-9 rational scroll $S$ which has 8 double points and is smooth otherwise.
\end{thm}

Here a \emph{double point} means a \emph{non-normal ordinary double point}. It's a point where the surface has two branches that meet transversally.

The idea in proving Theorem \ref{mainStep} is as follows:

Degree-9 scrolls in $\lP^5$ form a component $\mathcal{H}_9$ in the associated Hilbert scheme. Let $\mathcal{H}_9^8\subset\mathcal{H}_9$ parametrize scrolls with 8 isolated singularities. By definition (See Section \ref{sect:HilbSS}) an element $\overline{S}\in\mathcal{H}_9^8$ is non-reduced. We use $S$ to denote its underlying variety.

Let $U_{42}\subset|\fO_{\lP^5}(3)|$ be the locus of special cubic fourfolds with discriminant 42. Consider the incidence variety
\[
	\mathcal{Z} = \left\{(\overline{S},X)\in\mathcal{H}_9^8\times U_{42}:S\subset X\right\}.
\]
Then there is a diagram
\[\xymatrix{
	&\mathcal{Z}\ar_{p_1}[ld]\ar^{p_2}[rd]&&\\
	\mathcal{H}_9^8&&U_{42}\ar[r]&\mathcal{C}_{42}
}\]

Theorem \ref{mainStep} is proved by showing that $p_2$ is dominant. Two main ingredients in the proof are:
\begin{itemize}
\item Constructing an explicit example.
\item Estimating the dimension of the Hilbert scheme parametrizing singular scrolls.
\end{itemize}

Section \ref{sect:pSS} provides an introduction of rational scrolls and the basic properties required in the proof. We construct an example in Section \ref{sect:constr} and then prove the main results in Section \ref{sect:C42}. The general description about the Hilbert schemes $\mathcal{H}_9^8\subset\mathcal{H}_9$ and the estimate of the dimensions are left to Section \ref{sect:HilbSS}.

Throughout the paper we will frequently deal with the rational map
\[
	\Lambda_Q: \lP^{D+1}\dashrightarrow\lP^N
\]
defined as the projection from some $(D-N)$-plane $Q$. Here $D$ and $N$ are positive integers such that $D+1\geq N\geq 3$. We will assume $D\geq N\geq 5$ when we are studying singular scrolls.

\medskip
\noindent {\bf Acknowledgments:}
I am grateful to my advisor, Brendan Hassett, for helpful suggestions and his constant encouragement. I also appreciate the helpful comments from Nicolas Addington. I'd like to give my special thanks to the referee who points out a mistake in an earlier version of the paper and provides suggestions on how to fix it. I am grateful for the support of the National Science Foundation through DMS-1551514.

\section{Preliminary: rational scrolls}\label{sect:pSS}
We provide a brief review of rational scrolls and introduce necessary terminologies and lemmas in this section.
\subsection{Hirzebruch surfaces}
Let $m$ be a nonnegative integer, and let $\mathscr{E}$ be a rank two locally free sheaf on $\lP^1$ isomorphic to $\fO_{\lP^1}\oplus\fO_{\lP^1}(m)$. The Hirzebruch surface $\F_m$ is defined to be the associated projective space bundle $\lP(\mathscr{E})$.

Let $f$ be the divisor class of a fiber, and let $g$ be the divisor class of a section, i.e. the divisor class associated with Serre's twisting sheaf $\fO_{\lP(\mathscr{E})}(1)$. The Picard group of $\F_m$ is freely generated by $f$ and $g$, and the intersection pairing is given by
\[\begin{array}{c|cc}
		&f	&g\\
	\hline
	f 	&0	&1\\
	g	&1	&m.
\end{array}\]
The canonical divisor is $K_{\F_m}=-2g+(m-2)f$.

Let $a$ and $b$ be two integers, and let $h=ag+bf$ be a divisor on $\F_m$. The ampleness and the very ampleness for $h$ are equivalent on $\F_m$, and it happens if and only if $a>0$ and $b>0$. \cite[Chapter V, \S2.18]{Har77}

\begin{lemma}\label{hHir}
Suppose the divisor $ag+bf$ is ample. We have
\[\begin{array}{ccl}
	h^0\left(\F_m\,,\,ag+bf\right)&=&(a+1)\left(\frac{1}{2}am+b+1\right),\\
	h^i\left(\F_m\,,\,ag+bf\right)&=&0\quad\mbox{for all}\enspace i>0.
\end{array}\]
\end{lemma}
These formulas appear in several places in the literature with slightly different details depending on the contexts, for example \cite[Prop. 2.3]{Laf02}, \cite[p.543]{BBF04}, and \cite[Lemma 2.6]{Cos06}. It can be proved by induction on the integers $a$ and $b$ or by applying the projection formula to the bundle map $\pi:\F_m\rightarrow\lP^1$.

\subsection{Deformations of Hirzebruch surfaces}
$\F_m$ admits a deformation to $\F_{m-2k}$ for all $m>2k\geq0$. More precisely, there exitsts a holomorphic family $\tau:\mathcal{F}\rightarrow\C$ such that $\mathcal{F}_0\cong\F_m$ and $\mathcal{F}_t\cong\F_{m-2k}$ for $t\neq0$. The family can be written down explicitly by the equation
\begin{equation}\label{defHir}
	\mathcal{F} = \left\{{x_0}^my_1-{x_1}^my_2+t{x_0}^{m-k}{x_1}^ky_0=0\right\}
	\subset\lP^1\times\lP^2\times\C,
\end{equation}
where $\left(\left[x_0,x_1\right],\left[y_0,y_1,y_2\right],t\right)$ is the coordinate of $\lP^1\times\lP^2\times\C$. \cite[p.205]{BPV84}

Generally, $\F_m$ admits an analytic versal deformation with a base manifold of dimension $h^1\left(\F_m,T_{\F_m}\right)$ by the following lemma.
\begin{lemma}\label{absDef}\cite[Lemma 1. and Theorem 4.]{Sei92}
	There is a natural isomorphism $H^1\left(\F_m,T_{\F_m}\right)\cong H^1\left(\lP^1,\fO_{\lP^1}(-m)\right)$. We also have $H^2\left(\F_m,T_{\F_m}\right)=0$.
\end{lemma}

Let $\mathcal{E}=\fO_{\lP^1}\oplus\fO_{\lP^1}(m)$ be the underlying locally free sheaf of $\F_m$. It is straightforward to compute that ${\rm Ext^1}_{\lP^1}\left(\mathcal{E},\mathcal{E}\right)\cong H^1\left(\lP^1,\fO_{\lP^1}(-m)\right)$, so there is a natural isomorphism
\begin{equation}\label{natIsom}
	H^1\left(\F_m,T_{\F_m}\right)\cong{\rm Ext^1_{\lP^1}}\left(\mathcal{E},\mathcal{E}\right),
\end{equation}
by Lemma \ref{absDef}. The elements of the group ${\rm Ext^1}_{\lP^1}\left(\mathcal{E},\mathcal{E}\right)$ are in one-to-one correspondence with the deformations of $\mathcal{E}$ over the dual numbers $D_t\cong\frac{\C[t]}{(t^2)}$. \cite[Th. 2.7]{Har10} Thus (\ref{natIsom}) says that the infinitesimal deformation of $\F_m$ can be identified with the infinitesimal deformation of its underlying locally free sheaf.

Every element in
$
	{\rm Ext^1}_{\lP^1}\left(\mathcal{E},\mathcal{E}\right)\cong
	{\rm Ext^1}_{\lP^1}\left(\fO_{\lP^1}(m),\fO_{\lP^1}\right)
$
is represented by a short exact sequence
\[
	0\rightarrow\fO_{\lP^1}\rightarrow\fO_{\lP^1}(k)\oplus\fO_{\lP^1}(m-k)
	\rightarrow\fO_{\lP^1}(m)\rightarrow0
\]
for some $k$ satisfying $m>2k\geq0$. By tracking the construction of the correspondence in \cite[Th. 2.7]{Har10}, the above sequence corresponds to a coherent sheaf $\mathscr{E}$ on $\lP^1\times D_t$, flat over $D_t$, such that $\mathscr{E}_0\cong\mathcal{E}$ and $\mathscr{E}_t\cong\fO_{\lP^1}(k)\oplus\fO_{\lP^1}(m-k)$ for $t\neq0$. So it induces a flat family $\mathcal{F}$ of Hirzebruch surfaces over $D_t$ such that $\mathcal{F}_0\cong\F_m$ and $\mathcal{F}_t\cong\F_{m-2k}$ for $t\neq0$.

\subsection{Rational normal scrolls}
Let $u$ and $v$ be positive integers with $u\leq v$ and let $N=u+v+1$. Let $P_1$ and $P_2$ be complementary linear subspaces of dimensions $u$ and $v$ in $\lP^N$. Choose rational normal curves $C_1\subset P_1$, $C_2\subset P_2$, and an isomorphism $\varphi:C_1\rightarrow C_2$. Then the union of the lines $\bigcup_{p\in C_1}\overline{p\,\varphi(p)}$ forms a smooth surface $S_{u,v}$ called a \emph{rational normal scroll of type} $(u,v)$. The line $\overline{p\,\varphi(p)}$ is called a \emph{ruling}. When $u<v$, we call the curve $C_1\subset S_{u,v}$ the \emph{directrix} of $S_{u,v}$.

A rational normal scroll of type $(u,v)$ is uniquely determined up to projective isomorphism. In particular, each $S_{u,v}$ is projectively equivalent to the one given by the parametric equation
\begin{equation}\label{stdRNS}
\begin{array}{ccc}
	\C^2	&\longrightarrow	&\lP^N\\
	(s,t)		&\longmapsto		&(1,s,...,s^u,t,st,...,s^vt).
\end{array}
\end{equation}
One can check by this expression that a hyperplane section of $S_{u,v}$ which doesn't contain a ruling is a rational normal curve of degree $u+v$. It easily follows that $S_{u,v}$ has degree $D=u+v$.

The rulings of $S_{u,v}$ form a rational curve in $\lG(1,N)$ the Grassmannian of lines in $\lP^N$. By using (\ref{stdRNS}), we can parametrize this curve as
\begin{equation}\label{paraRul}
\begin{array}{ccc}
	\C	&\longrightarrow	&\lG(1,N)\\
	s	&\longmapsto		&\left(\begin{array}{cccccccc}
		1&s&...&s^u&0&0&...&0\\
		0&0&...&0&1&s&...&s^v
	\end{array}\right),
\end{array}
\end{equation}
where the matrix on the right represents the line spanned by the row vectors.

The embedding $S_{u,v}\subset\lP^N$ can be seen as the Hirzebruch surface $\F_{v-u}$ embedded in $\lP^N$ through the complete linear system $|g+uf|$. Conversely, every nondegenerate, irreducible and smooth surface of degree $D$ in $\lP^{D+1}$ isomorphic to $\F_{v-u}$ must be $S_{u,v}$.\cite[p.522-525]{GH94}

It's not hard to compute that $H^1\left(T_{\lP^N}|_{\F_{u-v}}\right)=0$ under the above embedding. Combining this with the rigidity result, it implies that every abstract deformation of $\F_{v-u}$ can be lifted to an embedded deformation as a family of rational normal scrolls in $\lP^N$. \cite[Remark 20.2.1]{Har10} We conclude this as the following lemma:
\begin{lemma}\label{embDef}
	For $m>2k\geq0$, let $\mathcal{F}$ be an abstract deformation of Hirzebruch surfaces such that $\mathcal{F}_0\cong\F_m$ and $\mathcal{F}_t\cong\F_{m-2k}$ for $t\neq0$. Then $\mathcal{F}$ can be realized as an embedded deformation $\mathcal{S}$ in $\lP^{D+1}$ with $\mathcal{S}_0\cong S_{u,v}$ and $\mathcal{S}_t\cong S_{u+k,v-k}$ for $t\neq0$, where $D=u+v$, and $u\leq v$ are any positive integers satisfying $v-u=m$.
\end{lemma}

\subsection{Rational scrolls}
\begin{defn}
We call a surface $S\subset\lP^N$ a rational scroll (or a scroll) of type $(u,m+u)$ if it is the image of a Hirzebruch surface $\F_m$ through a birational morphism defined by an $N$-dimensional subsystem $\textfrak{d}\subset|g+uf|$ for some $u>0$.
\end{defn}

Equivalently, $S\subset\lP^N$ is a rational scroll of type $(u,v)$ either if it is a rational normal scroll $S_{u,v}$, or if it is the projection image of $S_{u,v}\subset\lP^{D+1}$ from a $(D-N)$-plane disjoint from $S_{u,v}$. Here $D=u+v$ is the degree of $S_{u,v}$ as well as the degree of $S$. In the latter case, we also call a line on $S$ a ruling if its preimage is a ruling on $S_{u,v}$.

The following lemma computes the cohomology groups of the normal bundle for an arbitrary embedding of a Hirzebruch surface into a projective space.

\begin{lemma}\label{h0h1}
Let $\iota:\F_m\hookrightarrow\lP^N$ be an embedding with image $S$ and $\iota^*\fO_{\lP^N}(1)\cong\fO_S(h)$, where $h=ag+bf$ with $a>0$ and $b>0$. Let $N_{S/\lP^N}$ be the normal bundle of $S$ in $\lP^N$, then
\[
	h^0(S,N_{S/\lP^N})= (N+1)(a+1)(\frac{am}{2}+b+1)-7
\]
and $h^i(S,N_{S/\lP^N})=0,\forall i>0$. Especially, if $S$ is a smooth scroll of degree $D$, then the formula for $h^0$ reduces to
\[
	h^0(S,N_{S/\lP^N})= (N+1)(D+2)-7.
\]
\end{lemma}
\begin{proof}
The short exact sequence
\begin{equation}\label{h0h11}
	0\rightarrow T_S\rightarrow T_{\lP^N}|_{S}\rightarrow N_{S/\lP^N}\rightarrow0
\end{equation}
has the long exact sequence
\[\begin{array}{ccccccccc}
	0&\rightarrow&H^0(S,T_S)&\rightarrow&H^0(S,T_{\lP^N}|_{S})&\rightarrow&H^0(S,N_{S/\lP^N})&&\\
	&\rightarrow&H^1(S,T_S)&\rightarrow&H^1(S,T_{\lP^N}|_{S})&\rightarrow&H^1(S,N_{S/\lP^N})&&\\
	&\rightarrow&H^2(S,T_S)&\rightarrow&H^2(S,T_{\lP^N}|_{S})&\rightarrow&H^2(S,N_{S/\lP^N})&\rightarrow&0.
\end{array}\]
In order to calculate the dimensions in the right, we need the dimensions in the first two columns.

For the middle column, we can restrict the Euler exact sequence $$0\rightarrow\fO_{\lP^N}\rightarrow\fO_{\lP^N}(1)^{\oplus(N+1)}\rightarrow T_{\lP^N}\rightarrow0$$ to $S$ and obtain
\[
	0\rightarrow\fO_S\rightarrow\fO_S(h)^{\oplus(N+1)}\rightarrow T_{\lP^N}|_S\rightarrow0.
	\footnote{Tor$_i^{\fO_{\lP^N}}(\fO_S,\mathscr{Ful98})=0$ for all $i>0$ and locally free sheaf $\mathscr{Ful98}$, so the Euler exact sequence keeps exact after the restriction.}
\]
Lemma \ref{hHir} confirms that $h^i(S,\fO_S(h))=0$ for $i>0$, so we have
\[
	0\rightarrow H^0(S,\fO_S)\rightarrow H^0(S,\fO_S(h))^{\oplus(N+1)}\rightarrow H^0(S,T_{\lP^N}|_S)\rightarrow0
\]
from the associated long exact sequence while the other terms are all vanishing. It follows that
\[\begin{array}{rcl}
	h^0(S,T_{\lP^N}|_S) &=& (N+1)h^0(S,\fO_S(h)) - h^0(S,\fO_S)\\
	&\stackrel{\ref{hHir}}{=}& (N+1)(a+1)(\frac{am}{2}+b+1)-1.
\end{array}\]

For the first column, one can use the Hirzebruch-Riemann-Roch formula to compute that $\chi(T_S)=6$. We also have $h^2(S,T_S)=0$ by Lemma \ref{absDef}. Thus $h^0(S,T_S)-h^1(S,T_S)=\chi(T_S)=6$.

Collecting the above results, the long exact sequence for (\ref{h0h11}) now becomes
\[\begin{array}{ccccccccc}
	0&\rightarrow&H^0(S,T_S)&\rightarrow&H^0(S,T_{\lP^N}|_{S})&\rightarrow&H^0(S,N_{S/\lP^N})&&\\
	&\rightarrow&H^1(S,T_S)&\rightarrow&0&\rightarrow&H^1(S,N_{S/\lP^N})&&\\
	&\rightarrow&0&\rightarrow&0&\rightarrow&H^2(S,N_{S/\lP^N})&\rightarrow&0.
\end{array}\]
Therefore we have $h^i(S,N_{S/\lP^N})=0,\forall i>0$, and
\[\begin{array}{rcl}
	h^0(S,N_{S/\lP^N}) &=& h^0(S,T_{\lP^N}|_{S})-\chi(T_S)\\
	&=& (N+1)(a+1)(\frac{am}{2}+b+1)-7.
\end{array}\]

When $S$ is a rational scroll, we have $h=g+bf$. Then the formula is obtained by inserting $a=0$ and $D=h^2=m+2b$ into the equation.
\end{proof}

\subsection{Isolated singularities on rational scrolls}
The singularities on a rational scroll are all caused from projection by definition, so we assume $D\geq N$. We also assume $N\geq 5$.

Let $S\subset\lP^N$ be a rational scroll under $S_{u,v}\subset\lP^{D+1}$ and let $q:S_{u,v}\rightarrow S$ be the projection. A point $p\in S$ is singular if and only if one of the following situations occurs
\begin{itemize}
	\item There are two distinct rulings $l$, $l'\subset S_{u,v}$ such that $p\in q(l)\cap q(l')$.
	\item There is a ruling $l\subset S_{u,v}$ such that $p\in q(l)$ and the map $q$ is ramified at $l$.
\end{itemize}

Suppose that $S$ has isolated singularities, i.e. the singular locus of $S$ has dimension zero. Then each singular point is set-theoretically the intersection of two or more rulings. Let $m$ be the number of the ruling which passes through any of the singular points. Then the number of singularities on $S$ is counted as ${m\choose 2}$.

Note that $S_{u,v}$ is cut out by quadrics so every secant line intersects $S_{u,v}$ in exactly two points transversally. Let $T(S_{u,v})\subset S(S_{u,v})$ respectively be the tangent and the secant varieties of $S_{u,v}$. Then every $x\in S(S_{u,v})-T(S_{u,v})$ belongs to one of the two conditions
\begin{enumerate}
	\item The point $x$ lies on one and only one secant line.
	\item\label{sec2} The point $x$ lies on two secant lines. Let $Z_2\subset S(S_{u,v})$ denote the union of such points.
\end{enumerate}

\begin{lemma}\label{noMoreSec}
	The subset $Z_2\neq\emptyset$ if and only if $u=2$. In this situation, $Z_2\cong\lP^2$ and each $x\in Z_2$ lies on infinitely many secant lines.
\end{lemma}
\begin{proof}
We retain the notation used in constructing $S_{u,v}$ throughout the proof.

Let $x\in Z_2$ be any point. First we claim that the intersection points of $S_{u,v}$ with the union of the two secants described in (\ref{sec2}) lie on four distinct rulings.

The intersection points don't lie on two rulings because any two distinct rulings are linearly independent.\cite[Exercise 8.21]{Har95} If they lie on three rulings, then the projection to $P_2$ would be a trisecant line of $C_2$. But this is impossible because $C_2$ has degree $v\geq\lceil\frac{D}{2}\rceil\geq2$. Hence the claim holds.

The claim admits a rational normal curve $C\subset S_{u,v}$ (either sectional or residual) of degree $\geq u$ passing through the four intersection points.\cite[Example 8.17]{Har95} This imposes a non-trivial linear relation on four distinct points on $C$, which forces $C$ to be either a line or a conic. If $C$ is a line then $C$ coincides with the two secant lines, which is impossible. Hence $C$ must be a conic..

It follows that $u\leq\deg C\leq2$. If $u=1$, then the conic $C$ would dominate $C_2$ through the projection from $P_1$. However, this cannot happen since $C_2$ has degree $v=D-u\geq4$. Hence $u=2$. In this condition, $C$ can only be the directrix since $u<v$. It follows that the $Z_2$ coincides with the 2-plane spanned by $C$ and each point of $Z_2$ lies on infinitely many secants.

Conversely, $u=2$ implies that $Z_2$ contains the 2-plane spanned by $C_2$. By the same arguement above they coincide and every $x\in Z_2$ lies on infinitely many secants.
\end{proof}

Assume that $S$ is the projection of $S_{u,v}$ from a $(D-N)$-plane $Q\subset\lP^{D+1}$.
\begin{cor}\label{isoSingEq}
The scroll $S$ is singular along $r$ points if and only if $Q$ intersects $S(S_{u,v})$ in $r$ points away from $T(S_{u,v})\cup Z_2$.
\end{cor}
\begin{proof}
Assume that $S$ has isolated singularities. Recall that the number $r$ counts the number of the pair $(l,l')$ of distinct rulings on $S_{u,v}$ such that $q(l)$ intersects $q(l')$ in one point. (Different pairs might intersect in the same point.) It then counts the number of the unique line joining $l$, $l'$ and $Q$. By Lemma \ref{noMoreSec}, each $x\in S(S_{u,v})$ away from $T(S_{u,v})\cup Z_2$ lies on a unique secant. Thus it is the same as the number of the intersection between $Q$ and $S(S_{u,v})$ away from $T(S_{u,v})\cup Z_2$.
\end{proof}

In the end we provide a criterion for $S$ to have isolated singularities when $u=1$. This is going to be used in proving Proposition \ref{fourSq}.

\begin{prop}\label{noCSing}
Assume $u=1$. If $Q\cap T(S_{1,v})=\emptyset$, then $S$ has isolated singularities.
\end{prop}
\begin{proof}
If $Q$ intersects $S(S_{1,v})$ in points, then the proposition follows from Corollary \ref{isoSingEq}.

Assume $Q\cap S(S_{1,v})$ contains a curve $\Gamma$. We are going to show that $\Gamma$ intersects $T(S_{1,v})$ nontrivially which then contradicts our hypothesis.

Let $f$ be the fiber class of $S_{1,v}$. Then the linear system $|f|$ parametrizes the rulings of $S_{1,v}$. For distinct $l,l'\in|f|$, the linear span of $l$ and $l'$ is a 3-space $P_{l,l'}\subset S(S_{1,v})$. Consider the incidence correspondence
\[
	\varmathbb{S} = \{(x,l+l')\in\lP^{D+1}\times|2f|:x\in P_{l,l'}\}.
\]
Observe that $\varmathbb{S}$ is a $\lP^3$-bundle over $|2f|\cong\lP^2$ via the second projection
\[
	p_2:\varmathbb{S}\rightarrow|2f|.
\]
On the other hand, the image of $\varmathbb{S}$ under the first projection
\[
	p_1:\varmathbb{S}\rightarrow\lP^{D+1}
\]
is the secant variety $S(S_{1,v})$. Consider the diagonal
\[
	\Delta:=\{2l:l\in|f|\}\subset|2f|.
\]
It's easy to see that the tangent variety $T(S_{1,v})\subset S(S_{1,v})$ is the image of $p_2^{-1}(\Delta)$ via the first projection.

If $\Gamma\nsubset P_{l,l'}$ for all $l+l'$, then the curve $p_1^{-1}(\Gamma)$ is mapped to a curve in $|2f|$ which intersects $\Delta$ nontrivially. It follows that $\Gamma\cap T(S_{u,v})\neq\emptyset$.

Suppose $\Gamma\subset P_{l,l'}$ for some $l+l'$. The directrix $C_1$ is a line in $P_{l,l'}$ by hypothesis, so we have
\[
	T(S_{u,v})\cap P_{l,l'} = P\cup P'
\]
where $P$ and $P'$ are the 2-planes spanned by $C_1$ and $l$ and $l'$, respectively. So $\Gamma$ and $T(S_{u,v})$ has a nontrivial intersection in $P_{l,l'}$.
\end{proof}

\section{Construction of singular scrolls in ${\bf P}^5$}\label{sect:constr}
This section provides a construction of singular scrolls in $\lP^5$ of type $(1,v)$ with isolated singularities. The construction actually relates the existence of the singular scrolls to the solvability of a four-square equation as follows:

\begin{prop}\label{fourSq}
Assume $v\geq4$. There exists a rational scroll in $\lP^5$ of type $(1,v)$ with isolated singularities which has at least $r$ singularities if there are four odd integers $a\geq b\geq c\geq d>0$ satisfying
	\begin{enumerate}
		\item $8r+4=a^2+b^2+c^2+d^2$,
		\item $a+b+c\leq 2v-3$.
	\end{enumerate}
\end{prop}

We use the construction to produce an explicit example which can be manipulated by a computer algebra system. With the help of a computer we prove that
\begin{prop}\label{SinX}
	There is a degree-9 rational scroll $S\subset\lP^5$ which has eight isolated singularities and smooth otherwise such that
	\begin{enumerate}
		\item\label{SinX1} $h^0(\lP^5,\I_S(3))=6$, where $\I_S$ is the ideal sheaf of $S$ in $\lP^5$.
		\item\label{SinX2} $S$ is contained in a smooth cubic fourfold $X$.
		\item\label{SinX3} $S$ deforms in $X$ to the first order as a two dimensional family.
		\item\label{SinX4} $S$ is also contained in a singular cubic fourfold $Y$. 
	\end{enumerate}
\end{prop}

We introduce the construction first and prove Proposition \ref{SinX} in the end. Recall that, with a fixed rational normal scroll $S_{1,v}\subset\lP^{D+1}$, every degree $D$ scroll $S\subset\lP^5$ of type $(1,v)$ is the projection of $S_{1,v}$ from $P^\perp$ for some $P\in\lG(5,D+1)$.

\subsection{Plane $k$-chains}
Let $k$ be a positive integer. It can be proved by induction that $k$ distinct lines in a projective space intersect in at most ${k\choose2}$ points counted with multiplicity, and the maximal number is attained exactly when the $k$ lines span a 2-plane.

\begin{defn}
	Let $k\geq1$ be an integer. We call the union of $k$ distinct lines which span a 2-plane a plane $k$-chain. Let $W\subset\lP^N$ be the union of a finite number of lines. A plane $k$-chain in $W$ is called maximal if it is not a subset of a plane $k'$-chain in $W$ for some $k'>k$.
\end{defn}

Let $S\subset\lP^5$ be a singular scroll with isolated singularities. There's a subset $W\subset S$ consisting of a finite number of rulings defined by
\begin{equation}\label{rulConfig}
	W = \bigcup l:\mbox{$l$ is a ruling passing through a singular point on $S$.}
\end{equation}
By Zorn's lemma, $W$ can be expressed as
\[
	W = \bigcup_{i=1}^n K_i: \mbox{$K_i$ is a maximal plane $k_i$-chain with $k_i\geq2$.}
\]

If two plane $k$-chains share more than one line, then they must lie on the same 2-plane. In particular, both of them can not be maximal. Therefore, for distinct maximal plane $k$-chains $K_i$ and $K_j$ in $W$, we have either $K_i\cap K_j=\emptyset$ or $K_i\cap K_j=l$ a single ruling. It follows that the number of singularities on $S$ equals $\sum_{i=1}^n {k_i\choose2}$ since a plane $k$-chain contributes ${k\choose2}$ singularities.

Let $l_1,...,l_k\subset S_{1,v}$ be $k$ distinct rulings which span a subspace $P_{l_1,...,l_k}\subset\lP^{D+1}$. The images of the rulings form a plane $k$-chain on $S$ through projection if and only if $P_{l_1,...,l_k}$ is projected onto a 2-plane in $\lP^5$. Parametrize the rulings as in (\ref{paraRul}) with $l_j=l_j(s_j)$, $j=1,...,k$. Then $P_{l_1,...,l_k}$ is spanned by the row vectors of the following $(k+2)\times(D+2)$ matrix
\begin{equation}\label{P1tok}
	P(s_1,...,s_k)=
	\left(\begin{array}{cccccc}
		1&s_1&0&0&...&0\\
		1&s_2&0&0&...&0\\
		0&0&1&s_1&...&s_1^v\\
		&&\vdots&&&\\
		0&0&1&s_k&...&s_k^v
	\end{array}\right).
\end{equation}

The projection $S_{1,v}\rightarrow S$ is restricted from a linear map
\[
	\Lambda: \lP^{D+1}\dashrightarrow\lP^5.
\]
Suppose $\Lambda$ is represented by a $(D+2)\times6$ matrix
\[
	\Lambda = 
	\left(\begin{array}{cccccc} v_1&v_2&v_3&v_4&v_5&v_6\end{array}\right),
\]
where $v_1,...,v_6$ are vectors in $\lP^{D+1}$. Then $P_{l_1,...,l_k}$ is projected onto a 2-plane if and only if the $(k+2)\times6$ matrix
\[
	P(s_1,...,s_k)\cdot\Lambda
\]
has rank three.

\subsection{Control the number of singularities}
Let $r$ be a non-negative integer. We introduce a method to find a projection $\Lambda$ which maps $S_{1,v}$ to a singular scroll $S$ with isolated singularities. The method allows us to control the number of singularities such that it is bounded below by $r$. For simplicity, we consider only the cases when the configuration $W\subset S$ defined in (\ref{rulConfig}) consists of four disjoint maximal plane $k$-chains.

We start by picking distinct rulings on $S_{u,v}$ and produce four matrices $P_1$, $P_2$, $P_3$, and $P_4$ as in (\ref{P1tok}). Suppose $P_i$ consists of $k_i$ rulings. Note that $P_i$ contribute ${k_i\choose 2}$ singularities if its rulings are mapped to a plane $k_i$-chain. Thus we also assume that
\begin{equation}\label{singNum}
	r={k_1\choose2}+{k_2\choose2}+{k_3\choose2}+{k_4\choose2}.
\end{equation}
Here we allow $k_i=1$ which means that $P_i$ consists of a single ruling and thus contributes no singularity.

Consider $\Lambda = \left(\begin{array}{cccccc} v_1&v_2&v_3&v_4&v_5&v_6\end{array}\right)$ as an unknown. Let $P$ be the 5-plane spanned by $v_1,...,v_6$. We are going to construct $\Lambda$ satisfying
\begin{enumerate}
	\item\label{ld1} ${\rm rk}\left(P_i\cdot\Lambda\right)=3,\quad i=1,2,3,4$.
	\item\label{ld2} $P^\perp\,\cap\,T(S_{1,v})=\emptyset$.
\end{enumerate}
Note that (\ref{ld1}) makes the number of isolated singularities $\geq r$, while (\ref{ld2}) confirms that no curve singularity occurs. We divide the construction into two steps:

\vspace{1.5mm}
\noindent\emph{Step 1. Find $v_1$, $v_2$, $v_3$ and $v_4$ to satisfy (\ref{ld1}).}

Consider each $P_i$ as a linear map by multiplication from the left. We are trying to find independent vectors $v_1$, $v_2$, $v_3$ and $v_4$ such that for each $P_i$ three of them are in the kernel while the remaining one isn't. The four vectors arranged in this way contribute exactly one rank to each $P_i\cdot\Lambda$. In the next step, $v_5$ and $v_6$ will be general vectors in $\lP^{D+1}$ satisfying some open conditions. This contributes two additional ranks to each $P_i\cdot\Lambda$, which makes ($\ref{ld1}$) true.

Under the standard parametrization for $S_{1,v}\subset\lP^{D+1}$, the underlying vector space of $\lP^{D+1}$ can be decomposed as $A\oplus B$ with $A$ representing the first 2 coordinates and $B$ representing the last $v+1$ coordinates. With this decomposition, the matrix $P$ in (\ref{P1tok}) can be decomposed into two Vandermonde matrices
\[
	P^A=
	\left(\begin{array}{cc}
		1&s_1\\
		1&s_2
	\end{array}\right)
	\quad\mbox{and}\quad
	P^B=
	\left(\begin{array}{cccc}
		1&s_1&...&s_1^v\\
		&\vdots&&\\
		1&s_k&...&s_k^v
	\end{array}\right).
\]
Note that $\ker P=\ker P^B$. So we can search for the vectors from $\ker P^B$.

In our situation, we have four matrices ${P_1}^B$, ${P_2}^B$, ${P_3}^B$, ${P_4}^B$ which have four kernels
$\ker{P_1}^B$, $\ker{P_2}^B$, $\ker{P_3}^B$, and $\ker{P_4}^B$, respectively. By the assumption $k_i\leq v$ and the fact that a Vandermonde matrix has full rank, each $\ker {P_i}^B$ is a codimension $k_i$ subspace of $B$.

Now we want to pick $v_1,...,v_4$ from $B$ such that each $\ker{P_i}^B$ contains exactly three of the four vectors, i.e. we want
\begin{equation}\label{vec3-1}
	\left|\,\ker{P_i}^B\cap\{v_1,v_2,v_3,v_4\}\,\right|=3,\quad\mbox{for}\enspace i=1,2,3,4.
\end{equation}
One way to satisfy (\ref{vec3-1}) is to pick $v_i$ from
\begin{equation}\label{ker3-1}
	\left(\bigcap_{j\neq i}\ker{P_j}^B\right)-\ker{P_i}^B,\quad\mbox{for}\enspace i=1,2,3,4.
\end{equation}
The sets in (\ref{ker3-1}) are nonempty if and only if
\[
	\dim\left(\ker{P_\alpha}^B\cap\ker{P_\beta}^B\cap\ker{P_\gamma}^B\right)\geq1
\]
for all distinct $\alpha,\beta,\gamma\in\{1,2,3,4\}$. This is equivalent to 
\begin{equation}\label{4v}
	k_\alpha+k_\beta+k_\gamma\leq v,\enspace\mbox{for distinct}\;\alpha,\beta,\gamma\in\{1,2,3,4\}.
\end{equation}
So we have to include (\ref{4v}) as one of our assumptions.

\vspace{1.5mm}
\noindent\emph{Step 2. Adjust $v_1,...,v_4$ and then pick $v_5$ and $v_6$ to satisfy (\ref{ld2}).}

\begin{lemma}
	Let ${v_i}^\perp$ be the hyperplane in $\lP^{D+1}$ orthogonal to $v_i$. The four vectors $v_1,...,v_4$ can be chosen generally such that $\bigcap_{i=1}^4{v_i}^\perp$ intersects $T(S_{1,v})$ only in the directrix of $S_{1,v}$.
\end{lemma}
\begin{proof}
Parametrize the rational normal curve $C=S_{1,v}\cap\lP(B)$ by
\[
	\theta(s)=(0,0,1,s,...,s^v).
\]
Then the standard parametrization (\ref{stdRNS}) can be written as
\[
	(1,s,0,...,0)+t\theta(s).
\]
Let $a$ and $b$ be the parameters for the tangent plane over each point. Then the tangent variety $T(S_{1,v})$ has the parametric equation
\[\begin{array}{l}
	(1,s,0,...,0)+t\theta(s)+a\left[(0,1,0,...,0)+t\frac{d\theta}{ds}(s)\right]+b\theta(s)\\
	= (1,s+a,0,...,0)+(t+b)\theta(s)+ta\frac{d\theta}{ds}(s).
\end{array}\]

Each point on $T(S_{1,v})$ lying in $\bigcap_{i=1}^4{v_i}^\perp$ is a common zero of the equations
\begin{equation}\label{viIntTan}
	(t+b)\left(\theta(s)\cdot v_i\right)+ta\left(\frac{d\theta}{ds}(s)\cdot v_i\right)=0,\quad i=1,2,3,4.
\end{equation}
By considering $(t+b)$ and $ta$ as variables, (\ref{viIntTan}) becomes a system of linear equations given by the matrix
\[\left(\begin{array}{cccc}
	\theta(s)\cdot v_1&\theta(s)\cdot v_2&\theta(s)\cdot v_3&\theta(s)\cdot v_4\\
	\theta'(s)\cdot v_1&\theta'(s)\cdot v_2&\theta'(s)\cdot v_3&\theta'(s)\cdot v_4
\end{array}\right).\]
The matrix fails to be of full rank exactly when $s$ admits the existence of $\alpha,\beta\in\C$, $\alpha\beta\neq0$, such that
\begin{equation}\label{tanOfCurve}
	\left(\alpha\theta(s)+\beta\theta'(s)\right)\cdot v_i=0,\quad i=1,2,3,4.
\end{equation}
Note that (\ref{tanOfCurve}) has a solution if and only if $\bigcap_{i=1}^4{v_i}^\perp$ and the tangent variety $T(C)$ of $C$ intersect each other.

One can choose $v_2,v_3$ and $v_4$ in general from (\ref{ker3-1}) so that $\bigcap_{i=2}^4{v_i}^\perp$ is disjoint from $C$. This forces $\bigcap_{i=2}^4{v_i}^\perp$ to intersect $T(C)$ in either empty set or points. By the properties of a rational normal curve, the hyperplane orthogonal to a point on $C$ contains no invariant subspace when one perturb the point. Hence, after necessary perturbation of the chosen rulings, one can choose $v_1$ from (\ref{ker3-1}) such that $\left(\bigcap_{i=1}^4{v_i}^\perp\right)\cap T(C)=\emptyset$. As a result, the equations in (\ref{viIntTan}) become independent, so the solutions are $t=b=0$ or $a=0$, $t=-b$. Both solutions form the directrix of $S_{1,v}$.
\end{proof}

With the above adjustment, we can pick $v_5$ and $v_6$ in general in $\lP^{D+1}$ so that the $(D-5)$-plane $Q = {v_1}^\perp\cap...\cap{v_6}^\perp$ has no intersection with $T(S_{1,v})$. Note that the projection defined by $\Lambda$ is the same as the projection from $Q$. By Proposition \ref{noCSing}, this projection produces a rational scroll with isolated singularities.

\begin{prop}
	There exists a rational scroll in $\lP^5$ of type $(1,v)$ with isolated singularities which has at least $r$ singularities if there are four positive integers $k_1\geq k_2\geq k_3\geq k_4$ satisfying (\ref{singNum}) and (\ref{4v}):
\[
	r={k_1\choose 2}+{k_2\choose 2}+{k_2\choose 2}+{k_2\choose 2}
	\quad\mbox{and}\quad
	k_1+k_2+k_3\leq v.
\]
\end{prop}

Proposition \ref{fourSq} is obtained by expanding the binomial coefficients followed by a change of variables.

\subsection{Proof of Proposition \ref{SinX}}
In the following we exhibit an explicit example which can be manipulated by a computer algebra system over characteristic zero. The main program used in our work is {\sc Singular} \cite{DGPS}.

Consider $\lP^{10}$ with homogeneous coordinate ${\bf x}=(x_0,...,x_{10})$. We define the rational normal scroll $S_{1,8}$ by the $2\times2$ minors of the matrix
\[\left(
	\begin{array}{ccccccccc}
	x_0&x_2&x_3&x_4&x_5&x_6&x_7&x_8&x_9\\
	x_1&x_3&x_4&x_5&x_6&x_7&x_8&x_9&x_{10}
	\end{array}
\right).\]

In order to project $S_{1,8}$ onto a rational scroll whose singular locus is zero dimensional and consists of at least eight singular points, we use the method introduced previously to construct a projection
\[
\Lambda =
\arraycolsep=2.5pt
\left(\begin{array}{c} v_1\\v_2\\v_3\\v_4\\v_5\\v_6\end{array}\right)^T=
\left(\begin{array}{rrrrrrrrrrr}
	0&0&0&120&-34&-203&91&70&-56&13&-1\\
	0&0&2880&5184&-2372&-2196&633&261&-63&-9&2\\
	0&0&0&480&304&-510&-339&30&36&0&-1\\
	0&0&0&144&36&-196&-49&56&14&-4&-1\\
	1&0&1&1&1&1&1&1&1&1&1\\
	0&1&1&0&0&0&0&0&0&0&0
\end{array}\right)^T
\]
Let ${\bf z}=(z_0,...,z_5)$ be the coordinate for $\lP^5$. Then the projection $\lP^{10}\dashrightarrow\lP^5$ defined by $\Lambda$ can be explicite written by
\[
	{\bf z} = {\bf x}\cdot\Lambda.
\]

Let $S$ be the image of $S_{1,8}$ under the projection. Due to the limit of the author's computer, we check that $S$ has eight singularities and smooth otherwise over the finite field of order 31. On the other hand, the double point formula implies that $S$ has eight double points if the singular locus is isolated. Hence the singularity of $S$ consists of eight double points over characteristic zero as required.

The generators of the ideal of $S$ contain six cubics, so property (\ref{SinX1}) is confirmed.  Properties (\ref{SinX2}) and (\ref{SinX4}) can be easily checked by examining the linear combinations of those cubics.

The final step is to varify property (\ref{SinX3}). Let $X\subset\lP^5$ be a smooth cubic containing $S$. Let $F_1(S)$ and $F_1(X)$ denote the Fano variety of lines on $S$ and $X$, respectively. Then it is equivalent to show that $F_1(S)$ deforms in $F_1(X)$ to the first order with dimension two.

Let $\lG(1,5)$ be the grassmannian of lines in $\lP^5$. Every element ${\bf b}\in\lG(1,5)$ is parametrized by a $2\times6$ matrix
\begin{equation}\label{coorGrass}
	\left(\begin{array}{c}
	{\bf b}_1\\{\bf b}_2
	\end{array}\right)
	= \left(\begin{array}{cccccc}
	b_{10}&b_{11}&b_{12}&b_{13}&b_{14}&b_{15}\\
	b_{20}&b_{21}&b_{22}&b_{23}&b_{24}&b_{25}
	\end{array}\right)
\end{equation}
where ${\bf b}_1$ and ${\bf b}_2$ are two vectors which span the line ${\bf b}$.

Let $P_X=P_X({\bf z})$ be the homogeneous polynomial defining $X$. Let $V$ be the 6-dimensional linear space underlying $\lP^5$. Consider $P_X$ as a symmetric function defined on $V\oplus V\oplus V$. Then $F_1(X)\subset\lG(1,5)$ is cut out by the four equations
\begin{equation}\label{FanoX}
	P_X({\bf b}_1,{\bf b}_1,{\bf b}_1),\;
	P_X({\bf b}_1,{\bf b}_1,{\bf b}_2),\;
	P_X({\bf b}_1,{\bf b}_2,{\bf b}_2),\;
	P_X({\bf b}_2,{\bf b}_2,{\bf b}_2).
\end{equation}

Consider the Fano variety of lines on $S_{1,8}$ as a rational curve $\lP^1\subset\lG(1,10)$ parametrized by
\[Q=\left(
\begin{array}{ccccccccccc}
	r&s&0&0&0&0&0&0&0&0&0\\
	0&0&r^8&r^7s&r^6s^2&r^5s^3&r^4s^4&r^3s^5&r^2s^6&rs^7&s^8
\end{array}
\right)\]
where $(r,s)$ is the homogeneous coordinate for $\lP^1$. Then $F_1(S)\subset\lG(1,5)$ is defined by the parametric equation
\[
	R=Q\cdot\Lambda.
\]

Now consider a $2\times6$ matrix $dR$ whose first row consists of arbitrary linear forms on $\lP^1$ while the second row consists of arbitrary 8-forms. The coefficients of those forms introduce $2\cdot6+9\cdot6=66$ variables $c_1,...,c_{66}$. Then an abstract first order deformation of $F_1(S)$ in $\lG(1,5)$ is given by
\[
	R+dR.
\]

Inserting $R+dR$ into (\ref{FanoX}) gives us four polynomials in $r$ and $s$ with coefficients in $c_1,...,c_{66}$. The linear parts of the coefficients form a system of linear equations in $c_1,...,c_{66}$ whose associated matrix has rank 53. Then the first order deformation of $F_1(S)$ in $F_1(X)$ appears as solutions of the system.

In addition to the 53 constraints contributed by the above linear equations, we also have
\begin{itemize}
\item 4 constraints from the ${\rm GL}(2)$ action on the coordinates (\ref{coorGrass}).
\item 3 constraints from the automorphism group of $\lP^1$.
\item 4 constraints from rescaling the four equations (\ref{FanoX}).
\end{itemize}
So $F_1(S)$ deforms in $F_1(X)$ to the first order with dimension $66-53-4-3-4=2$.

\section{Special cubic fourfolds of discriminant 42}\label{sect:C42}
This section proves that a generic special cubic fourfold $X\in\mathcal{C}_{42}$ has a unirational parametrization of odd degree and also that $\mathcal{C}_{42}$ is uniruled. We also provide a discussion in the end talking about the difficulty of generalizing our method to higher discriminants.

\subsection{The space of singular scrolls}
The Zariski closure of the locus for degree-9 scrolls forms a component $\mathcal{H}_9$ in the associated Hilbert scheme. Let $\mathcal{H}_9^8\subset\mathcal{H}_9$ be the closure of the locus parametrizing scrolls with 8 isolated singularities. By Propositions \ref{fourSq} and Theorem \ref{singCodim} we have the estimate:
\begin{cor}
	$\mathcal{H}_9^8$ has codimension at most 8 in $\mathcal{H}_9$.
\end{cor} 

Note that $\mathcal{H}_9^8$ parametrizes non-reduced schemes by definition. In the following, we use an overline to specify an element $\overline{S}\in\mathcal{H}_9^8$ and denote by $S$ its underlying reduced subscheme.

Let $U\subset|\fO_{\lP^5}(3)|$ be the locus parametrizing smooth cubic fourfolds. Define
\[
	\mathcal{Z}=\left\{(\overline{S},X)\in\mathcal{H}_9^8\times U:S\subset X\right\}.
\]
By Proposition \ref{SinX} there exists $(\overline{S},X)\in\mathcal{Z}$ such that $S$ has isolated singularities and $X$ is smooth.

The right projection $p_2:\mathcal{Z}\rightarrow U$ factors through $U_{42}$, the preimage of $\mathcal{C}_{42}$ in $U$. Indeed, by definition $S$ is the image of a rational normal scroll $F\subset\lP^{10}$ through a projection. Let $\epsilon:F\rightarrow X$ be the composition of the projection followed by the inclusion into $X$. Let ${[S]_X}^2$ is the self-intersection of $S$ in $X$. Then the number of singularities $D_{S\subset X}=8$ on $S$ satisfies the double point formula \cite[Th. 9.3]{Ful98}:
\[
	D_{S\subset X}=
	\frac{1}{2}\left({[S]_X}^2-\epsilon^*c_2(T_X)+c_1(T_F)\cdot\epsilon^*c_1(T_X)-c_1(T_F)^2+c_2(T_F) \right).
\]
By using this formula on can get ${[S]_X}^2=41$. Let $h_X$ be the hyperplane class of $X$. Then the intersection table for $X$ is
\[\begin{array}{c|cc}
			&h_X^2	&S\\
	\hline
	h_X^2 	&3		&9\\
	S		&9		&41.
\end{array}\]
So $X$ has discriminant $3\cdot41-9^2=42$.

\subsection{Odd degree unirational parametrizations}
\begin{thm}\label{dom}
Consider the diagram
\[
\xymatrix{
	&\mathcal{Z}\ar_{p_1}[ld]\ar^{p_2}[rd]&&\\
	\mathcal{H}_9^8&&U_{42}\ar[r]&\mathcal{C}_{42}
}
\]
\begin{enumerate}
	\item\label{dom1} $\mathcal{Z}$ dominates $U_{42}$. Therefore a general $X\in\mathcal{C}_{42}$ contains a degree-9 rational scroll with 8 isolated singularities and smooth otherwise.
	\item $\mathcal{C}_{42}$ is uniruled.
\end{enumerate}
\end{thm}
\begin{proof}
Let $(\overline{S},X)\in\mathcal{Z}$ be a pair satisfying Proposition \ref{SinX}. Then
\[
	h^0(\lP^5,\I_S(3))=6.
\]
On the other hand, the short exact sequence
\[
	0\rightarrow\I_S(3)\rightarrow\fO_{\lP^5}(3)\rightarrow\fO_S(3)\rightarrow0
\]
implies that
\begin{equation}\label{h0I(3)}
	h^0(\lP^5,\I_S(3))\geq h^0(\lP^5,\fO_{\lP^5}(3))-h^0(S,\fO_S(3)).
\end{equation}
Let $F\subset\lP^{10}$ be the preimage scroll of $S$. Then $H^0(S,\fO_S(3))$ consists of the sections in $H^0(F,\fO_F(3))$ which cannot distinguish the preimage of a singular point. We have $h^0(F,\fO_F(3))=58$ by Lemma \ref{hHir}, so $h^0(S,\fO_S(3))=58-8=50$. So the right hand side of (\ref{h0I(3)}) equals $56-50=6$. Thus $h^0(\lP^5,\I_S(3))$ attains a minimum.

The left projection $p_1:\mathcal{Z}\rightarrow\mathcal{H}_9^8$ has fiber $\lP H^0(\lP^5,\I_S(3))$ over all $\overline{S}\in\mathcal{H}_9^8$. Because the fiber dimension is an upper-semicontinuous function, there is an open subset $V\subset\mathcal{H}_9^8$ containing $\overline{S}$ such that $\mathcal{Z}$ is a $\lP^5$-bundle over $V$. We have $\dim\mathcal{H}_9=59$ by Proposition \ref{SSHilb}. Hence $\dim\mathcal{H}_9^8\geq59-8=51$ by Theorem \ref{singCodim}. Thus $\mathcal{Z}$ has dimension at least $51+5=56$ in a neighborhood of $(\overline{S},X)$.

By Proposition \ref{SinX} (\ref{SinX3}), $\mathcal{Z}$ has fiber dimension at most 2 over an open subset of $p_2\left(\mathcal{Z}\right)$ which contains $X$. Hence $p_2\left(\mathcal{Z}\right)$ has dimension at least $56-2=54$ in a neighborhood of $X$. On the other hand, $U_{42}$ is an irreducible divisor in $U$. In particular, $U_{42}$ has dimension 54. So $\mathcal{Z}$ must dominate $U_{42}$.

Next we prove the uniruledness of $\mathcal{C}_{42}$.

We already know that $\mathcal{Z}$ has an open dense subset $\mathcal{Z}^\circ$ isomorphic to a $\lP^5$-bundle over $V\subset\mathcal{H}_9^8$. If we can prove that the composition $\mathcal{Z}^\circ\xrightarrow{p_2}U_{42}\rightarrow\mathcal{C}_{42}$ does not factor through this bundle map, then the proof is done.

Let $(\overline{S},X)\in\mathcal{Z}^\circ$ be the pair as before. By Proposition \ref{SinX}, $S$ is also contained in a singular cubic $Y$. Assume that the map $\mathcal{Z}^\circ\rightarrow\mathcal{C}_{42}$ does factor through the bundle map instead. Then all of the cubics in $\lP H^0(\lP^5,\I_S(3))$ would be in the same $\lP{\rm GL}(6)$-orbit. In particular, the smooth cubic $X$ and the singular cubic $Y$ would be isomorphic, but this is impossible.
\end{proof}

\begin{prop}\label{varrho}\cite[Prop. 38]{Has16} \cite[Prop. 7.4]{HT01}
Let $X$ be a cubic fourfold and $S\subset X $ be a rational surface. Suppose $S$ has isolated singularities and smooth normalization, with invariants $D=\deg S$, section genus $g_H$, and self-intersection $\left<S,S\right>_X$. If
\begin{equation}\label{varrhoIneq}
	\varrho=\varrho(S,X):=\frac{D(D-2)}{2}+(2-2g_H)-\frac{\langle S,S\rangle_X}{2}>0,
\end{equation}
then $X$ admits a unirational parametrization $\rho:\lP^4\dashrightarrow X$ of degree $\varrho$.
\end{prop}

\begin{cor}
	A general $X\in\mathcal{C}_{42}$ has an unirational parametrization of degree 13.
\end{cor}
\begin{proof}
By Theorem \ref{dom} (\ref{dom1}), a general cubic fourfold $X\in\mathcal{C}_{42}$ contains a degree-9 scroll $S$ having 8 isolated singularties, with $\langle S,S\rangle_X=41$ and $g_H=0$. Thus $\varrho=\frac{9\cdot7}{2}+2-\frac{41}{2}=13$ by Proposition \ref{varrho}.
\end{proof}

\subsection{Problems in higher discriminants}
Let $\Delta\subset\Sigma^{[2]}$ denote the divisor parametrizing non-reduced subschemes. Recall that it is a $\lP^1$-bundle over $\Sigma$. Its fibers correspond to smooth rational curves of degree $2n+1$ in $F_1(X)$, where the polarization on $F_1(X)$ is induced from $\lG(1,5)$. Each rational scroll $S\subset X$ induced by these rational curves has the intersection product
\[\begin{array}{c|cc}
			&h_X^2	&S\\
	\hline
	h_X^2 	&3		&2n+1\\
	S		&2n+1	&2n^2+2n+1,
\end{array}\]
where $h_X$ is the hyperplane section class of $X$. One can compute by the double point formula that $S$ has $n(n-2)$ singularities provided that they are all isolated. \cite[Prop. 7.2]{HT01}

In order to obtain an odd degree unirational parametrization for a generic member in $\mathcal{C}_d$ by Proposition \ref{varrho}, we need the existence of a degree $2n+1$ scroll $S\subset\lP^5$ with isolated singularities which has $n(n-2)$ singularities and is contained in a cubic fourfold $X$. We also need an estimate on the dimension of the associated Hilbert scheme $\mathcal{H}_{2n+1}^{n(n-2)}$ which contains $S$.

Section \ref{sect:constr} builds up a method to find such $S$, but the existence of a cubic fourfold $X$ containing $S$ requires examination with a computer. This works well with $n=4$ because in this case a generic such $S$ is contained in a cubic hypersurface. However, the same phenomenon may fail when $n\geq5$. Indeed, the Hilbert scheme $\mathcal{H}_{2n+1}^{n(n-2)}$ of degree $2n+1$ scrolls with $n(n-2)$ singularities satisfies $\dim\mathcal{H}_{2n+1}^{n(n-2)}\geq-n^2+14n+11$ by Theorem \ref{singCodim} and Proposition \ref{SSHilb}. When $5\leq n\leq8$, $\dim\mathcal{H}_{2n+1}^{n(n-2)}\geq55$ the dimension of cubic hypersurfaces in $\lP^5$, so a generic $S\in\mathcal{H}_{2n+1}^{n(n-2)}$ is not in a cubic fourfold. We don't know what happens when $n\geq 9$, but working in this range involves tedious trial and error.

\noindent{\bf Question.}
Assume $n\geq2$. Let $S\subset\lP^5$ be a degree $2n+1$ rational scroll which has $n(n-2)$ isolated singularities and smooth otherwise. When is $S$ contained in a cubic fourfold?

\section{The Hilbert scheme of rational scrolls}\label{sect:HilbSS}
Let $N\geq3$ be an integer. The Hilbert polynomial $P_S$ for a degree $D$ smooth surface $S\subset\lP^N$ has the following form
\[
	P_S(x) = \frac{1}{2}Dx^2+\left(\frac{1}{2}D+1-\pi\right)x+1+p_a,
\]
where $\pi$ is the genus of a generic hyperplane section and $p_a$ is the arithmetic genus of $S$. \cite[V, Ex 1.2]{Har77}

We are interested in the case when $S$ is a rational scroll. In this case $\pi=p_a=0$, so
\[
	P_S(x) = \frac{D}{2}x^2+\left(\frac{D}{2}+1\right)x+1.
\]
Every smooth surface sharing the same Hilbert polynomial has $\pi=0$ and $p_a=0$ also and thus is rational. We denote by ${\rm Hilb}_{P_S}(\lP^N)$ the Hilbert scheme of subschemes in $\lP^N$ with Hilbert polynomial $P_S$.

The closure of the locus parametrizing degree $D$ scrolls forms a component $\mathcal{H}_D\subset{\rm Hilb}_{P_S}(\lP^N)$. We study this space by stratifying it according to the types of the scrolls. Recall that, by fixing a rational normal scroll $S_{u,v}\subset\lP^{D+1}$ where $D=u+v$, a rational scroll $S\subset\lP^N$ of type $(u,v)$ is either $S_{u,v}$ itself or the image of $S_{u,v}$ projected from a disjoint $(D-N)$-plane. We define $\mathcal{H}_{u,v}\subset\mathcal{H}_D$ as the closure of the subset consisting of smooth rational scrolls of type $(u,v)$. In this section, we will first show that
\begin{prop}\label{SSHilb}
	Assume $D+1\geq N\geq 3$.
	\begin{enumerate}
	\item\label{SSHilb1} $\mathcal{H}_D$ is generically smooth of dimension $(N+1)(D+2)-7$.
	\item\label{SSHilb2} $\mathcal{H}_{u,v}$ is unirational of dimension $(D+2)N+2u-4-\delta_{u,v}$,
	\end{enumerate}
where $\delta_{u,v}$ is the Kronecker delta. We also have
	\begin{enumerate}\setcounter{enumi}{2}
	\item\label{SSHilb3} $\mathcal{H}_{u,v}\subset\mathcal{H}_{u+k,v-k}$ for $0\leq 2k<v-u$, and $\mathcal{H}_{\lfloor\frac{D}{2}\rfloor,\lceil\frac{D}{2}\rceil}=\mathcal{H}_D$.
	\end{enumerate}
\end{prop}

When $D+1=N$, a generic element of $\mathcal{H}_{u,v}$ is projectively equivalent to a fixed rational normal scroll $S_{u,v}\subset\lP^{D+1}$. In this case $\mathcal{H}_{u,v}$ is birational to $\lP{\rm GL}(D+2)$ quotient by the stablizer of $S_{u,v}$.

When $D\geq N$, a generic element in $\mathcal{H}_{u,v}$ is the projection of $S_{u,v}$ from a $(D-N)$-plane. Note that $\mathcal{H}_{u,v}$ also records the scrolls equipped with embedded points along their singular loci. Such element occurs when the $(D-N)$-plane contacts the secant variety of $S_{u,v}$. We denote by $\mathcal{H}_{u,v}^r\subset\mathcal{H}_{u,v}$ the closure of the subset parametrizing the schemes such that the singular locus of each of the underlying varieties consists of $\geq r$ isolated singularities. Let $\mathcal{H}_D^r\subset\mathcal{H}_D$ denote the union of $\mathcal{H}_{u,v}^r$ through all possible types.

The main goal of this section is to prove the following theorem
\begin{thm}\label{singCodim}
	Assume $D\geq N\geq 5$, and assume the existence of a degree $D$ rational scroll with isolated singularities in $\lP^N$ which has at least $r$ singularities. Suppose $rN\leq(D+2)^2-1$, then $\mathcal{H}_D^r$ has codimension at most $r(N-4)$ in $\mathcal{H}_D$. Especially when $r=1$, $\mathcal{H}_D^1$ is unirational of codimension exactly $N-4$.
\end{thm}

\subsection{The component of rational scrolls}
Here we give a general picture of the component $\mathcal{H}_D$ and also prove Proposition \ref{SSHilb}. Note that Proposition \ref{SSHilb} (\ref{SSHilb1}) follows immediately from Lemma \ref{h0h1}.

As mentioned before, $\mathcal{H}_{u,v}$ is birational to $\lP{\rm GL}(D+2)$ when $D+1=N$. In order to study the case of $D\geq N$, we introduce the \emph{projective Stiefel variety}.
\begin{defn}\label{pStfl}
Let $V_{N+1}(\C^{D+2})={\rm GL(D+2)}/{\rm GL(D-N+1)}$ be the homogeneous space of $(N+1)$-frames in $\C^{D+2}$. The group $\C^*$ acts on $V_{N+1}(\C^{D+2})$ by rescaling, which induces a geometric quotient $\lV(N,D+1)$ that we call a projective Stiefel variety.
\end{defn}

$\lV(N,D+1)$ has a fiber structure over $\lG(N,D+1)$:
\[\begin{array}{cccl}
	\lP{\rm GL}(N+1)&\hookrightarrow&\lV(N,D+1)&\\
	&&\downarrow{\scriptstyle p}&\\
	&&\lG(N,D+1)&.
\end{array}\]
An element $\Lambda\in\lV(N,D+1)$ over $P\in\lG(N,D+1)$ can be expressed as a $(D+2)\times(N+1)$-matrix
\[
	\Lambda = \left(\begin{array}{cccc} v_1&v_2&\dots&v_{N+1}\end{array}\right)_{(D+2)\times(N+1)}
\]
up to rescaling, where $v_1,...,v_{N+1}$ are column vectors which form a basis of the underlying vector space of $P$. In particular, each $\Lambda\in\lV(N,D+1)$ naturally defines a projection $\cdot\Lambda:\lP^{D+1}\dashrightarrow\lP^N$ by multiplying the coordinates from the right.

Let $S_{u,v}\subset\lP^{D+1}$ be the rational normal scroll given by the standard parametrization (\ref{stdRNS}). When $D\geq N$, every rational scroll in $\mathcal{H}_{u,v}$ is the image of $S_{u,v}$ under the projection defined by some $\Lambda\in\lV(N,D+1)$. So there is a dominant rational map
\begin{equation}\label{grassH}
\begin{array}{cccc}
	\pi=\pi(S_{u,v}):&\lV(N,D+1)&\dashrightarrow&\mathcal{H}_{u,v}\\
	&\Lambda&\longmapsto&S_{u,v}\cdot\Lambda,
\end{array}
\end{equation}
where $S_{u,v}\cdot\Lambda$ is the rational scroll given by the parametric equation
\[\begin{array}{ccc}
	\C^2	&\longrightarrow	&\lP^N\\
	(s,t)		&\longmapsto		&(1,s,...,s^u,t,st,...,s^vt)\cdot\Lambda.
\end{array}\]

\begin{proof}[Proof of Proposition \ref{SSHilb} (\ref{SSHilb2})]
Both $\lP{\rm GL}(D+2)$ and $\lV(N,D+1)$ are rational quasi-projective varieties, so $\mathcal{H}_{u,v}$ is unirational either when $D+1=N$ or $D\geq N$ by the above construction. The formula for the dimension of $\mathcal{H}_{u,v}$ holds by \cite[Lemma 2.6]{Cos06}.
\end{proof}

\begin{proof}[Proof of Proposition \ref{SSHilb} (\ref{SSHilb3})]
By Lemma \ref{embDef}, there exists an embedded deformation $\mathcal{S}$ in $\lP^{D+1}$ over the dual numbers $D_t=\frac{\C[t]}{(t^2)}$ with $\mathcal{S}_0\cong S_{u,v}$ and $\mathcal{S}_t\cong S_{u+k,v-k}$ for $t\neq0$. For every rational scroll $S\in\mathcal{H}_{u,v}$, we can find a $\Lambda\in\lV(N,D+1)$ such that $S=S_{u,v}\cdot\Lambda$. Then $\mathcal{S}\cdot\Lambda$ defines an infinitesimal deformation of $S$ to a rational scroll of type $(u+k,v-k)$, which forces the inclusion $\mathcal{H}_{u,v}\subset\mathcal{H}_{u+k,v-k}$ to hold.

When $(u,v)=(\lfloor\frac{D}{2}\rfloor,\lceil\frac{D}{2}\rceil)$, i.e. when $u=v$ or $u=v-1$, we have $\dim\mathcal{H}_D=\dim\mathcal{H}_{u,v}=(N+1)(D+2)-7$ by Proposition \ref{SSHilb} (\ref{SSHilb1}) and (\ref{SSHilb2}). Because $\mathcal{H}_D=\bigcup_{u+v=D}\mathcal{H}_{u,v}$, we must have $\mathcal{H}_{\lfloor\frac{D}{2}\rfloor,\lceil\frac{D}{2}\rceil}=\mathcal{H}_D$.
\end{proof}

\subsection{Projections that produce one singularity}
We are ready to study the locus in $\mathcal{H}_D$ which parametrizes singular scrolls. Assume $D\geq N\geq 5$. Let us start from studying the projections that produce one singularity.

\begin{proof}[Notations \& Facts]\renewcommand{\qedsymbol}{}
Let $K$ and $L$ be any linear subspaces of $\lP^{D+1}$.
\begin{enumerate}
	\item We use the same symbol to denote a projective space and its underlying vector space. The dimension always means the projective dimension.
	\item Assume $K\subset L$, we write $K^{\perp L}$ for the orthogonal complement of $K$ in $L$. When $L=\lP^{D+1}$, we write $K^\perp$ instead of $K^{\perp\lP^{D+1}}$.
	\item $K+L$ means the space spanned by $K$ and $L$. We write it as $K\oplus L$ if $K\cap L=\{0\}$, and  write it as $K\oplus_\perp L$ if $K$ and $L$ are orthogonal to each other.
\end{enumerate}
The following two relations can be derived by linear algebra.
\begin{equation}\label{perpSum}(K\cap L)^\perp = K^\perp+L^\perp.\end{equation}
\begin{equation}\label{perpIn}(K\cap L)^{\perp K} = (K\cap L)^\perp\cap K.\end{equation}
\end{proof}

\begin{defn}
Let $l$ and $l'$ be a pair of distinct rulings on $S_{u,v}$, and let $P_{l,\,l'}$ be the 3-plane spanned by them. We define $\sigma(l,l')$ to be a subvariety of $\lG(N,D+1)$ by
\[
	\sigma(l,l') = \left\{P\in\lG(N,D+1)\,:\,\dim(P\cap P_{l,\,l'}^\perp)\geq N-3\right\}.
\]
\end{defn}

\begin{lemma}\label{sgll}
	Let $p:\lV(N,D+1)\rightarrow\lG(N,D+1)$ be the bundle map. Then $p^{-1}(\,\sigma(l,l')\,)\subset\lV(N,D+1)$ consists of the projections which produce singularities by making $l$ and $l'$ intersect.
\end{lemma}
\begin{proof}
Let $P\in\lG(N,D+1)$ and $\Lambda\in p^{-1}(P)$ be arbitrary. The target space of the projection map $\cdot\Lambda$ is actually $P$. Let $L\subset\lP^{D+1}$ be any linear subspace, then the image $L\cdot\Lambda$ is identical to $(P^\perp+L)\cap P$. On the other hand, (\ref{perpSum}) and (\ref{perpIn}) implies that
$
	(P\cap L^\perp)^{\perp P} = (P\cap L^\perp)^\perp\cap P = (P^\perp+L)\cap P.
$
Therefore,
\[\begin{array}{l}
	N-1= \dim P-1 = \dim (P\cap L^\perp) + \dim(P\cap L^\perp)^{\perp P}\\
	= \dim (P\cap L^\perp) + \dim\left((P^\perp+L)\cap P\right) = \dim (P\cap L^\perp) + \dim\left(L\cdot\Lambda\right).
\end{array}\]

With $L= P_{l,\,l'}$, the equation implies that
\[\begin{array}{rcl}
	\dim(P\cap P_{l,\,l'}^\perp)\geq N-3&\Leftrightarrow&
	\dim\left(P_{l,\,l'}\cdot\Lambda\right)\leq2.
\end{array}\]
It follows that
\[
	p^{-1}(\,\sigma(l,l')\,) = \left\{\Lambda\in\lV(N,D+1)\,:\,\dim(P_{l,\,l'}\cdot\Lambda)\leq2\right\}.
\]
The image $P_{l,\,l'}\cdot\Lambda\subset\lP^N$ lies in a plane if and only if $l$ and $l'$ intersect each other after the projection $\cdot\Lambda:\lP^{D+1}\dashrightarrow\lP^N$. As a consequence, every $\Lambda\in p^{-1}(\,\sigma(l,l')\,)$ defines a projection which produces a singularity by making $l$ and $l'$ intersect.
\end{proof}

\subsection{The geometry of the variety $\boldsymbol{\sigma(l,l')}$}
The properties of the singular scroll locus that we are interested in are the unirationality and the dimension. As a preliminary, we describe here the geometry of the variety $\sigma(l,l')$, which implies immediately the rationality of $\sigma(l,l')$ and also allows us to find its dimension easily.

Instead of studying $\sigma(l,l')$ alone, the geometry would be more apparent if we consider generally the linear subspaces in $\lP^{D+1}$ which satisfies a certain intersectional condition. Fix a $(D-3)$-plane $L\subset\lP^{D+1}$. For every $j\geq0$, we define 
\begin{equation}\label{sb}
	\sigma_j(L) = \left\{\,P\in\lG(N,D+1):\dim(P\cap L)\geq N-4+j\,\right\}.
\end{equation}
For example, $\sigma_0(L)=\lG(N,D+1)$, and $\sigma_1(P_{l,\,l'}^\perp)=\sigma(l,l')$. Note that $P\subset L$ or $L\subset P$ if $j\geq\min\left(4\,,\,D-N+1\right)$ in (\ref{sb}), so we have
\[\begin{array}{lrl}
	\sigma_j(L)\supsetneq\sigma_{j+1}(L)	&{\rm if}&0\leq j<\min\left(4\,,\,D-N+1\right),\\
	\sigma_j(L)=\sigma_{j+1}(L)			&{\rm if}&j\geq\min\left(4\,,\,D-N+1\right).
\end{array}\]
Define $\sigma^\circ_j(L) = \left\{P\in\lG(N,D+1)\,:\,\dim(P\cap L)=N-4+j\right\}$, then 
\[\begin{array}{lrl}
	\sigma^\circ_j(L)=\sigma_{j}(L)-\sigma_{j+1}(L)	&{\rm if}&0\leq j<\min\left(4\,,\,D-N+1\right),\\
	\sigma^\circ_j(L)=\sigma_{j}(L)				&{\rm if}&j=\min\left(4\,,\,D-N+1\right).
\end{array}\]

\begin{lemma}\label{blSgm}
	Assume $1\leq j<\min\left(4\,,\,D-N+1\right)$, then $\sigma_j(L)$ is singular along $\sigma_{j+1}(L)$ and smooth otherwise. The singularity can be resolved by a  $\lG(3-j\,,\,D-N+4-j)$-bundle over $\lG(N-4+j\,,\,D-3)$. Especially, $\sigma_j(L)$ is rational with codimension $j(N-3+j)$ in $\lG(N,D+1)$.
\end{lemma}
\begin{proof}
We define ${\bf G}_j(L)$ to be the fiber bundle
\[\begin{array}{ccc}
	\lG(3-j\,,\,D-N+4-j)&\,\hookrightarrow\,&{\bf G}_j(L)\\
	&&\downarrow\\
	&&\lG(N-4+j\,,\,L)
\end{array}\]
by taking $\lG(3-j\,,\,Q^\perp)$ as the fiber over $Q\in\lG(N-4+j\,,\,L)$. Apparently ${\bf G}_j(L)$ is smooth and rational. We denote an element of ${\bf G}_j(L)$ as $(Q,R)$, where $Q$ belongs to the base and $R$ belongs to the fiber over $Q$.

In the following, we will construct a birational morphism from ${\bf G}_j(L)$ to $\sigma_j(L)$, which determines the rationality and the codimension immediately. Then we will study the singular locus by analyzing the tangent cone to $\sigma_j(L)$ at a point on $\sigma_{j+1}(L)$.

\vspace{1.5mm}
\noindent\emph{Step 1. A birational morphism from ${\bf G}_j(L)$ to $\sigma_j(L)$.} 

Every $P\in\sigma^\circ_j(L)$ can be decomposed as $P=(P\cap L)\oplus_\perp(P\cap L)^{\perp P}$. Because $P\cap L\in\lG(N-4+j\,,\,L)$ and $(P\cap L)^{\perp P}$ is a $(3-j)$-plane in $(P\cap L)^\perp$, this induces a morphism
\[\begin{array}{cccc}
	\iota:&\sigma^\circ_j(L)&\longrightarrow&{\bf G}_j(L)\\
	&P&\longmapsto&\left(P\cap L\,,\,(P\cap L)^{\perp P}\right).
\end{array}\]
On the other hand, $Q\oplus_\perp R\in\sigma_{j}(L)$ for every $(Q,R)\in{\bf G}_j(L)$ since $\dim(Q\cap L)=N-4+j$ by definition. Thus there is a morphism
\begin{equation}\label{blSg}
\begin{array}{cccc}
	\epsilon:&{\bf G}_j(L)&\longrightarrow&\sigma_j(L)\\
	&(Q,R)&\longmapsto&Q\oplus_\perp R.
\end{array}
\end{equation}
Clearly, the composition $\epsilon\circ\iota$ is the same as the inclusion $\sigma^\circ_j(L)\subset\sigma_j(L)$. Therefore $\epsilon$ is a birational morphism.

The smoothness and rationality of ${\bf G}_j(L)$ implies that $\sigma^\circ_j(L)$ is smooth and that $\sigma_j(L)$ is rational. Moreover,
\[\begin{array}{l}
	\dim\sigma_j(L) = \dim{\bf G}_j(L)\\
	=(4-j)(D-N+1)+(N-3+j)(D-N+1-j)\\
	=(N+1)(D-N+1)-j(N-3+j)\\
	=\dim\lG(N,D+1)-j(N-3+j).
\end{array}\]
Hence $\sigma_j(L)$ has codimension $j(N-3+j)$ in $\lG(N,D+1)$.

\vspace{1.5mm}
\noindent\emph{Step 2. The tangent cones to $\sigma_j(L)$.}

Choose any $P\in\sigma_j(L)$ and fix a $\phi\in T_P\lG(N,D+1)\cong{\rm Hom}\left(P,P^\perp\right)$. Let $T_P{\sigma_j(L)}$ be the tangent cone to $\sigma_j(L)$ at $P$. By definition, $\phi\in T_P{\sigma_j(L)}$ if and only if the condition $\dim(P\cap L)\geq N-4+j$ is kept when $P$ moves infinitesimally in the direction of $\phi$, which is equivalent to the condition that $P\cap L$ has a subspace $Q$ of dimension $N-4+j$ such that $\phi(Q)\subset L$.

Consider the decomposition
\[
	P^\perp = (P^\perp\cap L)\oplus_\perp(P^\perp\cap L)^{\perp P^\perp}.
\]
Define 
\[
	\Gamma:{\rm Hom}\left(P,P^\perp\right)\rightarrow
	{\rm Hom}\left(P\cap L,(P^\perp\cap L)^{\perp P^\perp}\right)
\]
to be the composition of the restriction to $P\cap L$ followed by the right projection of the above decomposition.

For any subspace $Q\subset P\cap L$, $\phi(Q)\subset L$ if and only if $\phi(Q)\subset P^\perp\cap L$, if and only if $Q\subset\ker\Gamma(\phi)$. So $L$ has a subspace $Q$ of dimension $N-4+j$ such that $\phi(Q)\subset L$ if and only if the (projective) dimension of $\ker\Gamma(\phi)$ is at least $N-4+j$. Therefore,
\begin{equation}\label{TSgker}
	T_P{\sigma_j(L)} = \left\{\phi\in{\rm Hom}\left(P,P^\perp\right):\dim\left(\ker\Gamma(\phi)\right)\geq N-4+j\right\}.
\end{equation}

Note that $\sigma_j(L)$ is the disjoint union of $\sigma^\circ_{j+k}(L)$ for all $k$ satisfying
\[0\leq k\leq\min\left(4\,,\,D-N+1\right)-j.\]
Assume $P\in\sigma^\circ_{j+k}(L)$, i.e. $\dim(P\cap L) = N-4+j+k$, then (\ref{TSgker}) is equivalent to
\begin{equation}\label{TSgrk}
	T_P{\sigma_j(L)} = \left\{\phi\in{\rm Hom}\left(P,P^\perp\right):{\rm rk}\,\Gamma(\phi)\leq k\right\}.
\end{equation}
When $k=0$, the constraint becomes ${\rm rk}\,\Gamma(\phi)=0$, so $T_P{\sigma_j(L)} = \ker\Gamma$ is a vector space. This reflects the fact that $\sigma_j(L)$ is smooth on $\sigma^\circ_j(L)$ for all $j$. On the other hand, from the inequality
\[
	\dim(P\cap L)+\dim(P^\perp\cap L)\leq\dim(L)-1,
\]
we get
\[\begin{array}{l}
	\dim(P^\perp\cap L)\leq\dim(L)-\dim(P\cap L)-1\\
	 = (D-3)-(N-4+j+k)-1=D-N-j-k.\\
\end{array}\]
It follows that
\[\begin{array}{l}
	\dim\left((P^\perp\cap L)^{\perp P^\perp}\right) = \dim(P^\perp)-\dim(P^\perp\cap L)-1\\
	\geq (D-N)-(D-N-j-k)-1 = j+k-1.
\end{array}\]
So $\dim\left((P^\perp\cap L)^{\perp P^\perp}\right)\geq j+k-1\geq k$ once $k\geq1$. Under this condition, the linear combination of  members of rank $k$ in ${\rm Hom}\left(P\cap L,(P^\perp\cap L)^{\perp P^\perp}\right)$ can have rank exceeding $k$. So $T_P{\sigma_j(L)}$ can not be a vector space, thus $P$ is a singularity of $\sigma_j(L)$.
\end{proof}

Recall that $\sigma(l,l')=\sigma_1(P_{l,\,l'}^\perp)$, so Lemma \ref{blSgm} implies that
\begin{cor}\label{sgCodim}
	$\sigma(l,l')$ is rational with codimension $N-2$ in $\lG(N,D+1)$.
\end{cor}

\subsection{Families of the projections}
The singularities we have studied are those produced from the intersection of a fixed pair of distinct rulings. Now we are going to make use of the variety $\sigma(l,l')$ to control multiple singularities.

Let ${\lP^1}^{[2]}$ be the Hilbert scheme of two points on $\lP^1$ and $U\subset{\lP^1}^{[2]}$ be the open subset parametrizing reduced subschemes. On the rational normal scroll $S_{u,v}$, the set of $r$ pairs of distinct rulings
\[
	\left\{ (l_1+{l_1}',..., l_r+{l_r}') : l_i\neq{l_i}'\;\forall i \right\}
\]
is parametrized by $U^{\times r}$.

Let $\Sigma_r$ be a subset of $U^{\times r}\times\lG(N,D+1)$ defined by
\[
	\Sigma_r = \left\{ (l_1+{l_1}',...,l_r+{l_r}', P)\in U^{\times r}\times\lG(N,D+1)\,:\,P\in\bigcap_{i=1}^r\sigma(l_i,{l_i}')\right\}.
\]
Let $p_1$ be the left projection and $p_2$ the right projection. Then there is a diagram
\[\begin{array}{ccccc}
	\bigcap_{i=1}^r\sigma(l_i,{l_i}')&\subset&\Sigma_r&\xrightarrow{\scriptstyle p_2}&\lG(N,D+1)\\
	\downarrow&&\!\!\!{\scriptstyle p_1}\downarrow\enspace&&\\
	(l_1+{l_1}',...,l_r+{l_r}')&\in&\enspace {\rm U}^{\times r}.\!\!&&
\end{array}\]
By Lemma \ref{sgll}, the image $p_2(\Sigma_r)$ consists of the $N$-planes such that the projections to them produce at least $r$ singularities. By the diagram above and Corollary \ref{sgCodim}, the codimension of $\Sigma_r$ in $U^{\times r}\times\lG(N,D+1)$ is at most $r(N-4)$. When $r=1$, $\Sigma_1$ is rational with codimension exactly $N-4$.

Our goal is to compute the dimension of $p_2(\Sigma_r)$, so we care about whether $p_2$ is generically finite onto its image or not. It turns out that the condition below is sufficient (See Lemma \ref{SgrCodim})
\begin{equation}\label{ass}
\parbox[c]{9.5cm}{
	There exists a rational scroll with isolated singularities $S\subset\lP^N$ of type $(u,v)$ which has at least $r$ singularities.
}
\end{equation}

By considering $S$ as the projection of $S_{u,v}$ from $P^\perp$ for some $N$-plane $P$, we can apply Corollary \ref{isoSingEq} to translate (\ref{ass}) into the equivalent statement:
\begin{equation}\tag{\ref{ass}'}\label{ass'}
\parbox[c]{9.5cm}{
There exists an $N$-plane $P$ such that $P^\perp$ intersects $S(S_{u,v})$ in $\geq r$ points away from $T(S_{u,v})\cup Z_2$.
}
\end{equation}

\vspace{1.5mm}
\begin{prop}\label{assHold}
	(\ref{ass}) holds for $r\leq D-N+1$.
\end{prop}
\begin{proof}
By \cite[Prop. 2.2]{CM96} and \cite[Example 19.10]{Har95}, $\deg\left(S(S_{u,v})\right)={D-2\choose 2}$. Since $\dim\left(S(S_{u,v})\right)=5$ and $T(S_{u,v})\cup Z_2$ form a proper closed subvariety of $S(S_{u,v})$, we can use Bertini's theorem to choose a $(D-4)$-plane $R$ which intersects $S(S_{u,v})$ in ${D-2\choose 2}$ points outside $T(S_{u,v})\cup Z_2$. It is easy to check that ${D-2\choose 2}\geq D-N+1$. Thus we can choose $D-N+1$ of the intersection points to span a $(D-N)$-plane $Q\subset R$. Then $P=Q^\perp$ satisfies the hypothesis.
\end{proof}

Unfortunately, Proposition \ref{assHold} doesn't cover the case $D=9$, $N=5$ and $r=8$ in our proof of the unirationality of discriminant 42 cubic fourfolds. In the following, we estimate the dimension of $p_2(\Sigma_r)$ under the assumption (\ref{ass}) and leave the construction of examples to Section \ref{sect:constr}.

\begin{lemma}\label{SgrCodim}
	Suppose (\ref{ass}) holds. Then $p_2(\Sigma_r)$ has codimension $\leq r(N-4)$ in $\lG(N,D+1)$. When $r=1$, $p_2(\Sigma_1)$ is rational of codimension exactly $N-4$.
\end{lemma}
\begin{proof}
Let $l$ and $l'$ be distinct rulings on $S_{u,v}$. We write $P_{l,\,l'}$ for the 3-plane spanned by them. Note that $S(S_{u,v}) = \overline{\,\bigcup_{l\neq l'} P_{l,\,l'}\,}$. Let $P$ be an $N$-plane satisfying (\ref{ass'}). Then there exists $r$ pairs of distinct rulings $\left(l_1,{l_1}'\right),...,\left(l_r,{l_r}'\right)$ such that $P^\perp$ and $P_{l_i,\,{l_i}'}$ intersect in exactly one point for each pair $\left(l_i,{l_i}'\right)$. This implies that $\dim\left(P^\perp+P_{l_i,\,{l_i}'}\right)\leq D-N+3$ for all $i$, which is equivalent to $\dim\left(P\cap P_{l_i,\,{l_i}'}^\perp\right)\geq N-3$ for all $i$ by (\ref{perpSum}). Hence $P\in\bigcap_{i=1}^r\sigma\left(l_i,{l_i}'\right)$, i.e. $P$ belongs to the image of $p_2$.

Suppose $P^\perp$ intersects $S(S_{u,v})$ in $m$ points, then the $\left(l_1+{l_1}',...,l_r+{l_r}'\right)$ in the preimage of $P$ is unique up to the choices of $r$ from $m$ pairs, the reordering of the $r$ pairs, and the transpositions of the rulings in a pair. Hence $p_2$ is generically finite with $\deg p_2={m\choose r}\cdot r!$. In particular, $\Sigma_r$ and $p_2(\Sigma_r)$ are equidimensional.

From $\dim\left(S(S_{u,v})\right)=5$ and our assumption that $N\geq5$, we are able to choose a $(D-N)$-plane which intersect $S(S_{u,v})$ in any one and exactly one point. Therefore, we can find $P$ so that $P^\perp$ intersects $S(S_{u,v})$ in one point outside $T(S_{u,v})\cup Z_2$. This provides an example of (\ref{ass}) for $m=r=1$. It follows that $p_2$ has degree one, and the image is rational since $\sigma\left(l_1,{l_1}'\right)$ is rational by Corollary \ref{sgCodim}.
\end{proof}

\subsection{Proof of Theorem \ref{singCodim}}
\begin{lemma}\label{singuvCodim}
	Assume $D\geq N\geq 5$, and assume the existence of a degree $D$ rational scroll $S\subset\lP^N$ with isolated singularities which has at least $r$ singularities. Then $\mathcal{H}_{u,v}^r$ has codimension at most $r(N-4)$ in $\mathcal{H}_{u,v}$. For $r=1$, $\mathcal{H}_{u,v}^1$ is unirational of codimension exactly $N-4$.
\end{lemma}
\begin{proof}
We have the following diagram
\[\xymatrix{
	&\lV(N,D+1)\ar_{p}[d]\ar@{-->}[r]^{\qquad\pi}&\mathcal{H}_{u,v}\\
	\Sigma_r\ar[r]^{p_2\qquad}&\lG(N,D+1).&
}\]
By definition, $\mathcal{H}_{u,v}^r = \pi\left(p^{-1}\left(p_2(\Sigma_r)\right)\right)$.

By Lemma \ref{SgrCodim}, $p_2(\Sigma_1)$ is rational, which implies that $\mathcal{H}_{u,v}^1$ is unirational.

It's clear that $p_2(\Sigma_r)$ and $p^{-1}\left(p_2(\Sigma_r)\right)$ have the same codimension. On the other hand, $p^{-1}\left(p_2(\Sigma_r)\right)$ and $\pi^{-1}\left(\pi\left(p^{-1}\left(p_2(\Sigma_r)\right)\right)\right)$ have the same dimension since both contain an open dense subset consisting of the projections which generate $r$ singularities, so the codimension of $p^{-1}\left(p_2(\Sigma_r)\right)$ is the same as its image through $\pi$. Therefore, $p_2(\Sigma_r)$ and $\mathcal{H}_{u,v}^r$ have the same codimension in their own ambient spaces, and the results follows from Lemma \ref{SgrCodim}.
\end{proof}

Lemma \ref{singuvCodim} is the special case of Theorem \ref{singCodim} when restricting to the locus of a particular type on the Hilbert scheme. The next lemma shows that a general $S\in\mathcal{H}_D^r$ deforms equisingularly between different types under the assumption $rN\leq(D+2)^2-1$. Hence the dimension estimate made by Lemma \ref{singuvCodim} can be extended regardless of the types.

\begin{lemma}\label{equiDim}
Assume (\ref{ass}) and $rN\leq(D+2)^2-1$, then $\mathcal{H}_{u,v}^r = \mathcal{H}_{u,v}\cap\mathcal{H}_{u+k,v-k}^r$ for $0\leq 2k<v-u$.
\end{lemma}
\begin{proof}
It is trivial that $\mathcal{H}_{u,v}^r\supset\mathcal{H}_{u,v}\cap\mathcal{H}_{u+k,v-k}^r$. To prove that $\mathcal{H}_{u,v}^r\subset\mathcal{H}_{u,v}\cap\mathcal{H}_{u+k,v-k}^r$, it is sufficient to show that a generic element in $\mathcal{H}_{u,v}^r$ deforms equisingularly to an element in $\mathcal{H}_{u+k,v-k}$.

If $(u,v)=\left(\lfloor\frac{D}{2}\rfloor,\lceil\frac{D}{2}\rceil\right)$ then there is nothing to prove, so we assume $(u,v)\neq\left(\lfloor\frac{D}{2}\rfloor,\lceil\frac{D}{2}\rceil\right)$. The elements satisfying (\ref{ass}) form an open dense subset of $\mathcal{H}_{u,v}^r$. Let $S\in\mathcal{H}_{u,v}^r$ be one of them, and assume $S$ is the image of $F\cong S_{u,v}\subset\lP^{D+1}$ projected from some $(D-N)$-plane $Q$. By hypothesis, $F$ has $r$ secants $\gamma_1$, ..., $\gamma_r$ incident to $Q$. Assume $\gamma_j\cap F=\{x_j,y_j\}$ for $j=1,...,r$.

$H^1\left(T_{\lP^{D+1}}|_F\right)=0$ by Lemma \ref{embDef}, so the short exact sequence
$
	0\rightarrow T_F\rightarrow T_{\lP^{D+1}}|_F\rightarrow N_{F/\lP^{D+1}}\rightarrow0
$
induces the exact sequence
\[
	0\rightarrow H^0\left(T_F\right)\rightarrow H^0\left(T_{\lP^{D+1}}|_F\right)
	\rightarrow H^0\left(N_{F/\lP^{D+1}}\right)\rightarrow H^1\left(T_F\right)\rightarrow0.
\]
By Lemma \ref{absDef}, $h^1\left(F,T_F\right)=h^1\left(\lP^1,\fO_{\lP^1}(u-v)\right)=v-u-1$, the same as the codimension of $\mathcal{H}_{u,v}$ in $\mathcal{H}_D$, thus a deformation normal to $\mathcal{H}_{u,v}$ is induced from an element in $H^1\left(T_F\right)$. In order to prove that the deformation is equisingular, it is sufficient to prove that for all $\mathcal{F}\in H^1\left(T_F\right)$ and its lift $\mathcal{S}\in H^0\left(N_{F/\lP^{D+1}}\right)$, there exists $\alpha\in H^0\left(T_{\lP^{D+1}}|_F\right)$ such that the vectors $\mathcal{S}(x_j)+\alpha(x_j)\in T_{\lP^{D+1},x_j}$ and $\mathcal{S}(y_j)+\alpha(y_j)\in T_{\lP^{D+1},y_j}$ keep $\gamma_j$ contact with $Q$ for $j=1,...,r$, so that $\mathcal{S}+\alpha$ is a lift of $\mathcal{F}$ representing an embedded deformation which preserves the incidence of the $r$ secants to $Q$.

Note that for arbitrary $p\in\lP^{D+1}$, the tangent space $T_{\lP^{D+1},p}\cong{\rm Hom}\left(p,p^\perp\right)\cong p^\perp$ can be considered as a subspace of $\lP^{D+1}$. We identify a point in $\lP^{D+1}$ with its underlying vector. Let $\gamma=\gamma_j$ for some $j$, and let $\{x,y\}=\gamma\cap S_{u,v}$ with $x=(x_1,...,x_{D+1})$ and $y=(y_1,...,y_{D+1})$. The condition that $\mathcal{S}(x)+\alpha(x)$ and $\mathcal{S}(y)+\alpha(y)$ keep $\gamma$ contact with $Q$ is equivalent to the condition that the set of vectors consisting of $x+\mathcal{S}(x)+\alpha(x)$, $y+\mathcal{S}(y)+\alpha(y)$ and the basis of $Q$ is not independent.

One can compute that $h^0\left(T_{\lP^{D+1}}|_F\right)=(D+2)^2-1$ by the Euler exact sequence
$
	0\rightarrow\fO_F\rightarrow\fO_F(1)^{\oplus(D+2)}\rightarrow T_{\lP^{D+1}}|_F\rightarrow0
$
and Lemma \ref{hHir}. Suppose $H^0\left(T_{\lP^{D+1}}|_F\right)$ has basis $e_1$, ..., $e_{(D+2)^2-1}$, we write the evaluation of $e_i$ at $p$ as $e_i(p)=(e_i(p)_1,...,e_i(p)_{D+1})$. Let $\alpha=\sum_{i\geq1}\alpha_ie_i$, $\mathcal{S}(x)=\sum_{i\geq1}c_ie_i(x)$ and $\mathcal{S}(y)=\sum_{i\geq1}d_ie_i(y)$, also write $Q=\left(q_{i,j}\right)$ as a $(D-N+1)\times(D+2)$-matrix. Then the dependence condition is equivalent to the condition that the $(D-N+3)\times(D+2)$-matrix
\[
	A_\gamma
	=\left(\begin{array}{c}
	\alpha_0x+\alpha_0\mathcal{S}(x)+\alpha(x)\\
	\alpha_0y+\alpha_0\mathcal{S}(y)+\alpha(y)\\
	Q
	\end{array}\right)\\
\]
\[
	=\left(\begin{array}{ccc}
	\alpha_0x_0+\sum\left(\alpha_0c_i+\alpha_i\right)e_i(x)_0&...&
	\alpha_0x_{D+1}+\sum\left(\alpha_0c_i+\alpha_i\right)e_i(x)_{D+1}\\
	\alpha_0y_0+\sum\left(\alpha_0d_i+\alpha_i\right)e_i(y)_0&...&
	\alpha_0y_{D+1}+\sum\left(\alpha_0d_i+\alpha_i\right)e_i(y)_{D+1}\\
	q_{0,0}&...&q_{0,D+1}\\
	\vdots&&\vdots\\
	q_{D-N,0}&...&q_{D-N,D+1}
	\end{array}\right)
\]
has rank at most $D-N+2$. Here we homogenize the first two rows by $\alpha_0$, so that the matrix defines a morphism
\[\begin{array}{cccc}
	{\bf A}_\gamma:&\lP\left(\C\oplus H^0\left(T_{\lP^{D+1}}|_F\right)\right)\cong\lP^{(D+2)^2-1}&
	\rightarrow&\lP^{(D-N+3)(D+2)-1}\\
	&(\alpha_0,...,\alpha_{(D+2)^2-1})&\mapsto&A_\gamma.
\end{array}\]

Let $M_{D-N+2}\subset\lP^{(D-N+3)(D+2)-1}$ be the determinantal variety of matrices of rank at most $D-N+2$. Then ${{\bf A}_\gamma}^{-1}\left(M_{D-N+2}\right)\subset\lP\left(\C\oplus H^0\left(T_{\lP^{D+1}}|_F\right)\right)$ is an irreducible and nondegenerate subvariety of codimension $N$, whose locus outside $\alpha_0=0$ parametrizes those $\alpha\in H^0\left(T_{\lP^{D+1}}|_F\right)$ such that $\mathcal{S}+\alpha$ preserves the incidence between $\gamma$ and $Q$.

It follows that the intersection $\bigcap_{j=1}^r{{\bf A}_{\gamma_j}}^{-1}\left(M_{D-N+2}\right)$ is nonempty by the hypothesis $rN\leq(D+2)^2-1$. Moreover, it is not contained in the hyperplane $\alpha_0=0$ for a generic $S\in\mathcal{H}_{u,v}^r$. Indeed, if this doesn't hold, then the limit case $\gamma_1=...=\gamma_r$ should also be inside the hyperplane $\alpha_0=0$. However, the intersection in that case is a multiple of a nondegenerate variety, a contradiction. As a result, for a generic $S\in\mathcal{H}_{u,v}^r$ we can find $\alpha$ from $\bigcap_{j=1}^r{{\bf A}_{\gamma_j}}^{-1}\left(M_{D-N+2}\right)$ which lies on $\{\alpha_0=1\}=H^0\left(T_{\lP^{D+1}}|_F\right)$, so that $\mathcal{S}+\alpha$ preserves the incidence condition between $\gamma_1$, ..., $\gamma_r$ and $Q$.
\end{proof}

Now we are ready to finish the proof of Theorem \ref{singCodim}.

Note that $\mathcal{H}_D^r=\bigcup_{u+v=D}\mathcal{H}_{u,v}^r$. By Lemma \ref{equiDim}
\[
	\bigcup_{u+v=D}\mathcal{H}_{u,v}^r = \bigcup_{u+v=D}\left(\mathcal{H}_{u,v}\cap\mathcal{H}_{\lfloor\frac{D}{2}\rfloor,\lceil\frac{D}{2}\rceil}^r\right) = \mathcal{H}_D\cap\mathcal{H}_{\lfloor\frac{D}{2}\rfloor,\lceil\frac{D}{2}\rceil}^r = \mathcal{H}_{\lfloor\frac{D}{2}\rfloor,\lceil\frac{D}{2}\rceil}^r.
\]
Therefore $\mathcal{H}_D^r=\mathcal{H}_{\lfloor\frac{D}{2}\rfloor,\lceil\frac{D}{2}\rceil}^r$, and the result follows from Lemma \ref{singuvCodim} with $(u,v) = (\lfloor\frac{D}{2}\rfloor,\lceil\frac{D}{2}\rceil)$.

\bigskip
\bibliography{CubicFourfold42_bib} 

\begin{thebibliography}{BPVdV84}

\bibitem[BBF04]{BBF04}
Edoardo Ballico, Cristiano Bocci, and Claudio Fontanari.
\newblock Zero-dimensional subschemes of ruled varieties.
\newblock {\em Cent. Eur. J. Math.}, 2(4):538--560, 2004.

\bibitem[BPVdV84]{BPV84}
W.~Barth, C.~Peters, and A.~Van~de Ven.
\newblock {\em Compact {C}omplex {S}urfaces}, volume~4 of {\em Ergebnisse der
  Mathematik und ihrer Grenzgebiete (3)}.
\newblock Springer-Verlag, Berlin, 1984.

\bibitem[BS83]{BS83}
S.~Bloch and V.~Srinivas.
\newblock Remarks on correspondences and algebraic cycles.
\newblock {\em Amer. J. Math.}, 105(5):1235--1253, 1983.

\bibitem[CJ96]{CM96}
Michael~L. Catalano-Johnson.
\newblock The possible dimensions of the higher secant varieties.
\newblock {\em Amer. J. Math.}, 118(2):355--361, 1996.

\bibitem[Cos06]{Cos06}
Izzet Coskun.
\newblock Degenerations of surface scrolls and the {G}romov-{W}itten invariants
  of {G}rassmannians.
\newblock {\em J. Algebraic Geom.}, 15(2):223--284, 2006.

\bibitem[DGPS15]{DGPS}
Wolfram Decker, Gert-Martin Greuel, Gerhard Pfister, and Hans Sch\"onemann.
\newblock {\sc Singular} {4-0-2} --- {A} computer algebra system for polynomial
  computations.
\newblock \url{http://www.singular.uni-kl.de}, 2015.

\bibitem[Ful98]{Ful98}
William Fulton.
\newblock {\em Intersection {T}heory}, volume~2 of {\em Ergebnisse der
  Mathematik und ihrer Grenzgebiete. 3. Folge. A Series of Modern Surveys in
  Mathematics}.
\newblock Springer-Verlag, Berlin, second edition, 1998.

\bibitem[GH94]{GH94}
Phillip Griffiths and Joseph Harris.
\newblock {\em Principles of algebraic geometry}.
\newblock Wiley Classics Library. John Wiley \& Sons, Inc., New York, 1994.
\newblock Reprint of the 1978 original.

\bibitem[Har77]{Har77}
Robin Hartshorne.
\newblock {\em Algebraic {G}eometry}.
\newblock Springer-Verlag, New York-Heidelberg, 1977.
\newblock Graduate Texts in Mathematics, No. 52.

\bibitem[Har95]{Har95}
Joe Harris.
\newblock {\em Algebraic {G}eometry}, volume 133 of {\em Graduate Texts in
  Mathematics}.
\newblock Springer-Verlag, New York, 1995.
\newblock A first course, Corrected reprint of the 1992 original.

\bibitem[Har10]{Har10}
Robin Hartshorne.
\newblock {\em Deformation {T}heory}, volume 257 of {\em Graduate Texts in
  Mathematics}.
\newblock Springer, New York, 2010.

\bibitem[Has00]{Has00}
Brendan Hassett.
\newblock Special cubic fourfolds.
\newblock {\em Compositio Math.}, 120(1):1--23, 2000.

\bibitem[Has16]{Has16}
Brendan Hassett.
\newblock Cubic fourfolds, {K}3 surfaces, and rationality questions.
\newblock In R.~Pardini and G.P. Pirola, editors, {\em Rationality Problems in
  Algebraic Geometry}, \rm CIME Foundation Subseries, Lecture Notes in
  Mathematics 2172, pages 29--66. Springer, 2016.

\bibitem[HT01]{HT01}
B.~Hassett and Y.~Tschinkel.
\newblock Rational curves on holomorphic symplectic fourfolds.
\newblock {\em Geom. Funct. Anal.}, 11(6):1201--1228, 2001.

\bibitem[Laf02]{Laf02}
Antonio Laface.
\newblock On linear systems of curves on rational scrolls.
\newblock {\em Geom. Dedicata}, 90:127--144, 2002.

\bibitem[Nue16]{Nue15}
Howard Nuer.
\newblock Unirationality of moduli spaces of special cubic fourfolds and {K}3
  surfaces.
\newblock In {\em Rationality problems in algebraic geometry}, volume 2172 of
  {\em Lecture Notes in Math.}, pages 161--167. Springer, Cham, 2016.

\bibitem[Sei92]{Sei92}
Wolfgang~K. Seiler.
\newblock Deformations of ruled surfaces.
\newblock {\em J. Reine Angew. Math.}, 426:203--219, 1992.

\bibitem[Voi13]{Voi13}
Claire Voisin.
\newblock Abel-{J}acobi map, integral {H}odge classes and decomposition of the
  diagonal.
\newblock {\em J. Algebraic Geom.}, 22(1):141--174, 2013.

\bibitem[Voi14]{Voi14}
Claire Voisin.
\newblock {\em Chow rings, decomposition of the diagonal, and the topology of
  families}, volume 187 of {\em Annals of Mathematics Studies}.
\newblock Princeton University Press, Princeton, NJ, 2014.

\bibitem[Voi17]{Voi15}
Claire Voisin.
\newblock On the universal {CH{$_0$}} group of cubic hypersurfaces.
\newblock {\em J. Eur. Math. Soc. (JEMS)}, 19(6):1619--1653, 2017.

\end{thebibliography}
\bibliographystyle{alpha}
\end{document}